\documentclass{article}
\usepackage{array,amsmath,amsthm}
\numberwithin{equation}{section}
\usepackage{amssymb}
\usepackage{indentfirst}
\usepackage{geometry}
\begin{document}

\newtheorem{thm}{Theorem}[section]
\newtheorem{cor}[thm]{Corollary}
\newtheorem{prop}[thm]{Proposition}
\newtheorem{conj}[thm]{Conjecture}
\newtheorem{lem}[thm]{Lemma}
\newtheorem{Def}[thm]{Definition}
\newtheorem{rem}[thm]{Remark}
\newtheorem{prob}[thm]{Problem}
\newtheorem{ex}{Example}[section]

\newcommand{\be}{\begin{equation}}
\newcommand{\ee}{\end{equation}}
\newcommand{\ben}{\begin{enumerate}}
\newcommand{\een}{\end{enumerate}}
\newcommand{\beq}{\begin{eqnarray}}
\newcommand{\eeq}{\end{eqnarray}}
\newcommand{\beqn}{\begin{eqnarray*}}
\newcommand{\eeqn}{\end{eqnarray*}}
\newcommand{\bei}{\begin{itemize}}
\newcommand{\eei}{\end{itemize}}

\newcommand{\pa}{{\partial}}
\newcommand{\V}{{\rm V}}
\newcommand{\R}{{\bf R}}
\newcommand{\K}{{\rm K}}
\newcommand{\e}{{\epsilon}}
\newcommand{\tomega}{\tilde{\omega}}
\newcommand{\tOmega}{\tilde{Omega}}
\newcommand{\tR}{\tilde{R}}
\newcommand{\tB}{\tilde{B}}
\newcommand{\tGamma}{\tilde{\Gamma}}
\newcommand{\fa}{f_{\alpha}}
\newcommand{\fb}{f_{\beta}}
\newcommand{\faa}{f_{\alpha\alpha}}
\newcommand{\faaa}{f_{\alpha\alpha\alpha}}
\newcommand{\fab}{f_{\alpha\beta}}
\newcommand{\fabb}{f_{\alpha\beta\beta}}
\newcommand{\fbb}{f_{\beta\beta}}
\newcommand{\fbbb}{f_{\beta\beta\beta}}
\newcommand{\faab}{f_{\alpha\alpha\beta}}

\newcommand{\pxi}{ {\pa \over \pa x^i}}
\newcommand{\pxj}{ {\pa \over \pa x^j}}
\newcommand{\pxk}{ {\pa \over \pa x^k}}
\newcommand{\pyi}{ {\pa \over \pa y^i}}
\newcommand{\pyj}{ {\pa \over \pa y^j}}
\newcommand{\pyk}{ {\pa \over \pa y^k}}
\newcommand{\dxi}{{\delta \over \delta x^i}}
\newcommand{\dxj}{{\delta \over \delta x^j}}
\newcommand{\dxk}{{\delta \over \delta x^k}}

\newcommand{\px}{{\pa \over \pa x}}
\newcommand{\py}{{\pa \over \pa y}}
\newcommand{\pt}{{\pa \over \pa t}}
\newcommand{\ps}{{\pa \over \pa s}}
\newcommand{\pvi}{{\pa \over \pa v^i}}
\newcommand{\ty}{\tilde{y}}
\newcommand{\bGamma}{\bar{\Gamma}}

\title {Elliptic Harnack inequality and its applications on Finsler metric measure spaces
\footnote{The first author is supported by the National Natural Science Foundation of China (12371051, 12141101, 11871126).}}
\author{ Xinyue Cheng, Liulin Liu  and Yu Zhang }
\date{}

\maketitle

\begin{abstract}
In this paper, we study the elliptic Harnack inequality and its applications on forward complete Finsler metric measure spaces under the conditions that the weighted Ricci curvature ${\rm Ric}_{\infty}$ has non-positive lower bound and the distortion $\tau$ is of linear growth, $|\tau|\leq ar+b$, where $a,b$ are some non-negative constants, $r=d(x_0,x)$ is the distance function for some point $x_{0} \in M$. We obtain an elliptic $p$-Harnack inequality for positive harmonic functions from a local uniform Poincar\'{e} inequality and a mean value inequality. As applications of the Harnack inequality, we derive the H\"{o}lder continuity estimate and a Liouville theorem for positive harmonic functions. Furthermore, we establish a gradient estimate for positive harmonic functions. 
\end{abstract}

\maketitle

\section{Introduction}
With the maturation of the theory of partial differential equations, geometric analysis has been fully developed in the last few decades and has become an important field in current geometric research. The study of Harnack inequality is one of the important research topics in geometric analysis. Elliptic Harnack inequality originated from Moser's work \cite{M1} in 1961. He gave the proof of a elliptic Harnack inequality for positive solutions of uniformly elliptic equations on a ball in $\mathbb{R}^n$ by using his famous iterative argument. As an application of the elliptic Harnack inequality,  Moser deduced related uniform estimate on the H\"{o}lder continuity. Later, Yau proved a Harnack inequality for positive harmonic functions on a ball in complete Riemannian manifolds under the condition that Ricci curvature has a non-positive lower bound in \cite{Yau}.  Anderson-Schoen obtained two elliptic Harnack-type inequalities for positive harmonic functions in complete Riemannian manifolds, from which they characterized Poisson heat function and Martin expression formula (\cite{AnS}).  Krylov and Safonov proved a elliptic Harnack inequality for non-negative solutions of an elliptic linear differential equation of order two and obtained the H\"{o}lder continuity estimate (see Theorem 4.41 in \cite{au}).

Let $M$ be a complete Riemannian manifold. One says that $M$ has the strong Liouville property if any solution $u$ of the Laplace equation $\Delta u=0$ on $M$ which is bounded below (or above) is a constant. Usually, the weak Liouville property refers to the same result for bounded solutions. In the classical case of harmonic functions in $\mathbb{R}^n$, the strong (and weak) Liouville property is satisfied and one can also prove that a harmonic function which has sublinear growth must be a constant. It is well known that Liouville-type properties follow from Harnack-type inequalities in Riemannian geometry (see \cite{SC}). For more details about elliptic Harnack inequality and its applications in Riemannian geometry, please refer to \cite{AnS,au,N,SC1,SC,Yau}.

It is natural to study and develop Harnack inequalities and the relevant theories on Finsler metric measure manifolds. In \cite{XiaC}, C. Xia proved the local gradient estimate for positive harmonic functions on forward complete, non-compact Finsler measure spaces under the condition that the weighted Ricci curvature Ric$_N$ has a non-positive lower bound. As applications, he obtained a Harnack inequality and two Liouville-type theorems on forward complete and non-compact Finsler manifolds with non-negative weighted Ricci curvature ${\rm Ric}_N$.
In \cite{XiaQ1},  Q. Xia gave the local and global gradient estimates for positive Finsler $p$-eigenfunctions on forward complete non-compact Finsler metric measure spaces with weighted Ricci curvature ${\rm Ric}_{N}\geq -K$ for some $K\geq 0$. As applications of local gradient estimate, she first obtained  a Harnack inequality for a positive $p (>1)$-eigenfunction when Finsler metric $F$ satisfies the uniform convexity and uniform smoothness. Then she obtained a Liouville property  for a $p$-eigenfunction bounded from below when  ${\rm Ric}_{N}\geq 0$.
Besides, Q. Xia  proved some Liouville theorems of non-negative $L^{p} (0<p \leq 1)$ subharmonic functions on forward complete Finsler measure spaces with finite reversibility and ${\rm Ric}_{N} \geq -K $ for some $K \geq 0$ (\cite{XiaQ2}).  On the other hand,  recently the first author and Feng derived local and global Harnack inequalities for positive harmonic functions on forward complete Finsler measure spaces under the conditions that ${\rm Ric}_{\infty} \geq K$ and $|\tau|\leq k$ for some $K\in \mathbb{R}$  and $k>0$ in \cite{CF}. Furthermore, they gave a gradient estimate for positive harmonic functions on  forward complete non-compact Finsler measure spaces equipped with a uniformly convex and uniformly smooth Finsler metric $F$ by using a the mean value inequality and the Harnack inequality (\cite{CF}).

For convenience to introduce our main results, we first give some necessary definitions and notations.  We always use $(M, F, m)$ to denote a Finsler manifold $(M, F)$ equipped with a smooth measure $m$ which we call a Finsler metric measure manifold (or Finsler measure space briefly). A Finsler measure space is not a metric space in usual sense because $F$ may be non-reversible, i.e., $F(x, y) \neq F(x, -y)$ may happen. This non-reversibility causes the asymmetry of the associated distance function. A Finsler metric $F$ on $M$ is said to be reversible if $F(x, -y)=F(x, y)$ for all $(x, y) \in TM$. Otherwise, we say $F$ is irreversible. In order to overcome the obstacles caused by irreversibility, Rademacher  defined the reversibility constant $\Lambda$ of $F$ by
\be
\Lambda:=\sup\limits_{(x, y)\in TM\setminus\{0\}}\frac{F(x,-y)}{F(x,y)}.
\ee
Obviously, $\Lambda\in [1,\infty]$ and $\Lambda=1$ if and only if $F$ is reversible (\cite{Ra}). On the other hand, Ohta extended the concepts of uniform smoothness and uniform convexity in Banach space theory into Finsler geometry and gave their geometric interpretation in \cite{Ohta}. We say that $F$ satisfies uniform convexity and uniform smoothness if there exist two uniform constants $0<\kappa^{*}\leq 1 \leq \kappa <\infty$ such that for $x\in M$, $V\in T_xM\setminus \{0\}$ and $y \in T_xM$, we have
\be
\kappa^{*}F^{2}(x,y)\leq g_{V}(y,y)\leq \kappa F^{2}(x,y),  \label{Fsc}
\ee
where $g_{V}$ is the weighted Riemann metric induced by $V$ on the tangent bundle of corresponding Finsler manifolds (see (\ref{weiRiem})). If $F$ satisfies the uniform smoothness and uniform convexity, then $\Lambda$ is finite with
\be
1\leq \Lambda \leq \min\{\sqrt{\kappa},\sqrt{1/ \kappa^*}\}.
\ee
$F$ is Riemannian if and only if $\kappa=1$ if and only if $\kappa^*=1$ (\cite{ChernShen, Ohta, Ra}). If $F$ is uniformly smooth and convex with (\ref{Fsc}), then $\left(g^{ij}\right)$ is uniformly elliptic in the sense that there exist two constants $\tilde{\kappa}=(\kappa^*)^{-1}$, $\tilde{\kappa}^*=\kappa^{-1}$ such that for $x \in M, \ \xi \in T^*_x M \backslash\{0\}$ and $\eta \in T_x^* M$, we have
\be
\tilde{\kappa}^* F^{* 2}(x, \eta) \leq g^{*i j}(x, \xi) \eta_i \eta_j \leq \tilde{\kappa} F^{* 2}(x, \eta). \label{F*sc}
\ee

For $x_1, x_2 \in M$, the distance from $x_1$ to $x_2$ is defined by
\be
d_{F}(x_{1}, x_{2}):=\inf _{\gamma} \int_{0}^{1} F(\gamma (t), \dot{\gamma}(t)) d t,
\ee
where the infimum is taken over all $C^1$ curves $\gamma:[0,1] \rightarrow M$ such that $\gamma(0)=$ $x_1$ and $\gamma(1)=x_2$. Note that $d_{F} \left(x_1, x_2\right) \neq d_{F} \left(x_2, x_1\right)$ unless $F$ is reversible.  The diameter of $M$ is defined by
\be
{\rm Diam}(M):=\sup _{x_{1}, x_{2} \in M} \{d_{F}(x_{1}, x_{2})\}.
\ee

Now we define the forward and backward geodesic balls of radius $R$ with center at $x_0\in M$ by
$$
B^{+}_{R}(x_0):=\{z\in M \ |\ d_{F}(x_0,z)<R\},\ \ \ B^{-}_{R}(x_0):=\{z \in M \ | \ d_{F}(z,x_0)<R\}.
$$
In the following, we will denote $B_{R}:=B^{+}_{R}(x_0)$ for some $x_{0}\in M$ for simplicity.

In this paper, we are mainly going to derive the elliptic $p$-Harnack inequality on forward complete Finsler metric measure spaces under the conditions that the weighted Ricci curvature ${\rm Ric}_{\infty}$ has non-positive lower bound and $|\tau|\leq ar+b$, where $\tau$ denotes the distortion of the Finsler metrics, $a,b$ are some non-negative constants, $r=d(x_0,x)$ is the distance function starting from some point $x_{0} \in M$. Then we will discuss some important applications of the elliptic $p$-Harnack inequality. Our first main result is given by the following theorem.

\begin{thm}\label{Harnack}
Let $(M, F, m)$ be an $n$-dimensional forward complete Finsler measure space with finite reversibility $\Lambda$. Assume that ${\rm Ric}_{\infty}\geq -K$ for some $K\geq 0$ and $|\tau|\leq ar+b$, where $a,b$ are some non-negative constants, $r=d_{F}(x_{0},x)$ is the distance function. Fix a constant $p\in (0,\infty)$.  If $u$ is a positive harmonic function on $B_R$, then for any $\delta \in (0,1)$, there exists a positive constant $C=C\left(n, \delta, b, p, \Lambda \right)$ depending on $n, \delta, b, p$ and $\Lambda$, such that
\begin{equation}
\sup\limits_{B_{\delta R}}u^p \leq e^{C(1+aR+KR^2)}\inf\limits_{B_{\delta R}}u^p.\label{harnackp}
\end{equation}
Here,  $B_{\delta R}$ is a concentric ball of radius $\delta R$ with $B_{R}$.
\end{thm}

In the following, we will give several important applications of the $p$-Harnack inequality. H\"{o}lder continuity is crucial for understanding the geometric and topological properties of manifolds. The following theorem characterizes H\"{o}lder continuity for positive harmonic functions on $B_R$.

\begin{thm}\label{Holder}
Let $(M, F, m)$ be an $n$-dimensional forward complete Finsler measure space equipped with a uniformly convex and uniformly smooth Finsler metric $F$. Assume that ${\rm Ric}_{\infty}\geq -K$ for some $K\geq 0$ and $|\tau|\leq ar+b$, where $a,b$ are some non-negative constants, $r=d_{F}(x_{0},x)$ is the distance function. If $u$ is a positive harmonic function on $B_R$, then, for any $\rho \in (0,1)$, there exist positive constants $\alpha=\alpha(n, a,b,\kappa,\kappa^*,R,K)$, such that
\begin{equation}
\sup\limits_{x_{1},x_{2}\in B_{\rho R}} \left\{\frac{|u(x_1)-u(x_2)|}{d_{F}(x_{1},x_{2})^\alpha}\right\} \leq 2^{\alpha} (1-\rho)^{-\alpha}\kappa^{\frac{\alpha}{2}}R^{-\alpha} \sup\limits_{B_R}\{u\}.\label{holder}
\end{equation}
\end{thm}

Further, as an application of Theorem \ref{Harnack}, we can obtain the following Liouville property for positive harmonic functions.

\begin{thm}\label{Liouville1}
Let $(M, F, m)$ be an $n$-dimensional forward complete non-compact Finsler measure space with finite reversibility $\Lambda$. Assume that ${\rm Ric}_{\infty}\geq 0$ and $|\tau|\leq b$, where $b$ is a non-negative constant. If $u$ is a positive harmonic function bounded from below on $M$, then $u$ is a constant.
\end{thm}

Finally, we give a gradient estimate for positive harmonic functions by the Harnack inequality given in Theorem \ref{Harnack}.

\begin{thm}\label{Grad}
Let $(M, F, m)$ be an $n$-dimensional forward complete non-compact Finsler measure space equipped with a uniformly convex and uniformly smooth Finsler metric $F$. Assume that ${\rm Ric}_{\infty}\geq -K$ for some $K \geq 0$ and $|\tau|\leq ar+b$, where $a,b$ are some non-negative constants, $r=d_{F}(x_{0},x)$ is the distance function. Let $u$ be a positive harmonic function on $M$, that is, $\Delta u=0$ in a weak sense on $M$. Then there exists a positive constant $C=C\left(n, b, \kappa, \kappa^*\right)$ depending only on $n$, $b$, the uniform constants $\kappa$ and $\kappa^*$, such that
\begin{equation}\label{grad}
\max\limits_{x\in M}\{F(x, \nabla \log u(x)), F(x, \nabla(-\log u(x)))\} \leq  e^{C (1+\frac{3}{2}a+K)}\left(1+\frac{K}{16}\right)^{n+4b}.
\end{equation}
\end{thm}

The paper is organized as follows. In Section \ref{Introd}, we will give some necessary definitions and notations. Then we will prove a Laplacian comparison theorem and a volume comparison theorem on forward complete Finsler metric measure manifolds under the condition that ${\rm Ric}_{\infty}\geq -K$ for some $K\geq 0$ and $|\tau|\leq ar+b$ in Section \ref{Volcom}. Next, we will prove a $L^p$ mean value inequality for non-negative subsolutions of a class of elliptic equations firstly, and then, we will give the proof of Theorem \ref{Harnack} in Section \ref{sharnack}. Finally, we will discuss the applications of $p$-Harnack inequality given in Theorem \ref{Harnack} and  prove Theorem \ref{Holder}, Theorem \ref{Liouville1} and Theorem \ref{Grad} in Section \ref{Appl}.

\section{Preliminaries}\label{Introd}
In this section, we briefly review some necessary definitions, notations and  fundamental results in Finsler geometry. For more details, we refer to \cite{BaoChernShen, ChernShen, Ohta1, Shen1}.

Let $M$ be an $n$-dimensional smooth manifold. A Finsler metric on manifold $M$ is a function $F: T M \longrightarrow[0, \infty)$  satisfying the following properties: (1) $F$ is $C^{\infty}$ on $TM\backslash\{0\}$; (2) $F(x,\lambda y)=\lambda F(x,y)$ for any $(x,y)\in TM$ and all $\lambda >0$; (3)  $F$ is strongly convex, that is, the matrix $\left(g_{ij}(x,y)\right)=\left(\frac{1}{2}(F^{2})_{y^{i}y^{j}}\right)$ is positive definite for any nonzero $y\in T_{x}M$. The pair $(M,F)$ is called a Finsler manifold and $g:=g_{ij}(x,y)dx^{i}\otimes dx^{j}$ is called the fundamental tensor of $F$. A non-negative function on $T^{*}M$ with analogous properties is called a Finsler co-metric. For any Finsler metric $F$, its dual metric
\be
F^{*}(x, \xi):=\sup\limits_{y\in T_{x}M\setminus \{0\}} \frac{\xi (y)}{F(x,y)}, \ \ \forall \xi \in T^{*}_{x}M. \label{co-Finsler}
\ee
is a Finsler co-metric. We define the reverse metric $\overleftarrow{F}$ of $F$ by $\overleftarrow{F}(x, y):=F(x,-y)$ for all $(x, y) \in T M$. It is easy to see that $\overleftarrow{F}$ is also a Finsler metric on $M$ and that a Finsler metric $F$ on $M$ is reversible if $\overleftarrow{F}(x, y)=F(x, y)$ for all $(x, y) \in T M$.

For a non-vanishing vector field $V$ on $M$, one can introduce the weighted Riemannian metric $g_V$ on $M$ given by
\be
g_V(y, w)=g_{ij}(x, V_x)y^i w^j  \label{weiRiem}
\ee
for $y,\, w\in T_{x}M$. In particular, $g_V(V,V)=F^2(V,V)$.

Let $(M,F)$ be a Finsler manifold of dimension $n$. The pull-back $\pi ^{*}TM$ admits a unique linear connection, which is called the Chern connection. The Chern connection $D$ is determined by the following equations
\beq
&& D^{V}_{X}Y-D^{V}_{Y}X=[X,Y], \label{chern1}\\
&& Zg_{V}(X,Y)=g_{V}(D^{V}_{Z}X,Y)+g_{V}(X,D^{V}_{Z}Y)+ 2C_{V}(D^{V}_{Z}V,X,Y) \label{chern2}
\eeq
for $V\in TM\setminus \{0\}$  and $X, Y, Z \in TM$, where
$$
C_{V}(X,Y,Z):=C_{ijk}(x,V)X^{i}Y^{j}Z^{k}=\frac{1}{4}\frac{\pa ^{3}F^{2}(x,V)}{\pa V^{i}\pa V^{j}\pa V^{k}}X^{i}Y^{j}Z^{k}
$$
is the Cartan tensor of $F$ and $D^{V}_{X}Y$ is the covariant derivative with respect to the reference vector $V$.

Given a non-vanishing vector field $V$ on $M$, the Riemannian curvature $R^V$ is defined by
$$
R^V(X, Y) Z=D_X^V D_Y^V Z-D_Y^V D_X^V Z-D_{[X, Y]}^V Z
$$
for any vector fields $X$, $Y$, $Z$ on $M$. For two linearly independent vectors $V, W \in T_x M \backslash\{0\}$, the flag curvature is defined by
$$
\mathcal{K}^V(V, W)=\frac{g_V\left(R^V(V, W) W, V\right)}{g_V(V, V) g_V(W, W)-g_V(V, W)^2}.
$$
Then the Ricci curvature is defined as
$$
\operatorname{Ric}(V):=F(x, V)^{2} \sum_{i=1}^{n-1} \mathcal{K}^V\left(V, e_i\right),
$$
where $e_1, \ldots, e_{n-1}, \frac{V}{F(V)}$ form an orthonormal basis of $T_x M$ with respect to $g_V$.

A $C^{\infty}$-curve $\gamma:[0,1] \rightarrow M$ is called a geodesic  if $F(\gamma, \dot{\gamma})$ is constant and it is locally minimizing. The exponential map $\exp_x: T_x M \rightarrow M$ is defined by $\exp_x(v)=\gamma(1)$ for $v \in T_x M$ if there is a geodesic $\gamma:[0,1] \rightarrow M$ with $\gamma(0)=x$ and $\dot{\gamma}(0)=v$. A Finsler manifold $(M, F)$ is said to be forward complete (resp. backward complete) if each geodesic defined on $[0, \ell)$ (resp. $(-\ell, 0])$ can be extended to a geodesic defined on $[0, \infty)$ (resp. $(-\infty, 0])$. We say $(M, F)$ is complete if it is both forward complete and backward complete. By Hopf-Rinow theorem on forward complete Finsler manifolds, any two points in $M$ can be connected by a minimal forward geodesic and the forward closed balls $\overline{B_R^{+}(p)}$ are compact (see \cite{BaoChernShen, Shen1}).

Let $(M, F, m)$ be an $n$-dimensional Finsler metric measure manifold with a smooth measure $m$. Write the volume form $dm$ of $m$ as $d m = \sigma(x) dx^{1} dx^{2} \cdots d x^{n}$. Define
\be\label{Dis}
\tau (x, y):=\ln \frac{\sqrt{{\rm det}\left(g_{i j}(x, y)\right)}}{\sigma(x)}.
\ee
We call $\tau$ the distortion of $F$. It is natural to study the rate of change of the distortion along geodesics. For a vector $y \in T_{x} M \backslash\{0\}$, let $\sigma=\sigma(t)$ be the geodesic with $\sigma(0)=x$ and $\dot{\sigma}(0)=y$.  Set
\be\label{S}
{\bf S}(x, y):= \frac{d}{d t}\left[\tau(\sigma(t), \dot{\sigma}(t))\right]|_{t=0}.
\ee
$\mathbf{S}$ is called the S-curvature of $F$ (\cite{ChernShen, shen}).

Let $Y$ be a $C^{\infty}$ geodesic field on an open subset $U \subset M$ and $\hat{g}=g_{Y}.$  Let
\be
d m:=e^{-\psi} {\rm Vol}_{\hat{g}}, \ \ \ {\rm Vol}_{\hat{g}}= \sqrt{{det}\left(g_{i j}\left(x, Y_{x}\right)\right)}dx^{1} \cdots dx^{n}. \label{voldecom}
\ee
It is easy to see that $\psi$ is given by
\be
\psi(x)= \ln \frac{\sqrt{\operatorname{det}\left(g_{i j}\left(x, Y_{x}\right)\right)}}{\sigma(x)} =\tau\left(x, Y_{x}\right), \label{Psi}
\ee
which is just the distortion of $F$ along $Y_{x}$ at $x\in M$ (\cite{ChernShen, Shen1}). Let $y := Y_{x}\in T_{x}M$ (that is, $Y$ is a geodesic extension of $y\in T_{x}M$). Then, by the definitions of the S-curvature, we have
\beqn
&&  {\bf S}(x, y)= Y[\tau(x, Y)]|_{x} = d \psi (y),  \\
&&  \dot{\bf S}(x, y)= Y[{\bf S}(x, Y)]|_{x} =y[Y(\psi)],
\eeqn
where $\dot{\bf S}(x, y):={\bf S}_{|m}(x, y)y^{m}$ and ``$|$" denotes the horizontal covariant derivative with respect to the Chern connection (\cite{shen, Shen1}). Further, the weighted Ricci curvatures are defined as follows (\cite{ChSh, Ohta1})
\beq
{\rm Ric}_{N}(y)&=& {\rm Ric}(y)+ \dot{\bf S}(x, y) -\frac{{\bf S}(x, y)^{2}}{N-n},   \label{weRicci3}\\
{\rm Ric}_{\infty}(y)&=& {\rm Ric}(y)+ \dot{\bf S}(x, y). \label{weRicciinf}
\eeq
We say that Ric$_N\geq K$ for some $K\in \mathbb{R}$ if Ric$_N(v)\geq KF^2(v)$ for all $v\in TM$, where $N\in \mathbb{R}\setminus \{n\}$ or $N= \infty$.

\begin{rem} On a Riemannian manifold $(M, g)$, the consideration of weighted measure of the form $\mathrm{e}^{-f} {\rm Vol}_{g}$ arises naturally in various situations, where $f$ is a smooth function and ${\rm Vol}_{g}$ is the standard volume element induced by the metric $g$.  In \cite{WeWy}, Wei-Wylie proved mean curvature and volume comparison results for Riemannian manifolds with a measure $\left(M, g, e^{-f} {\rm Vol}_{g} \right)$ when the $\infty$-Bakry-Emery Ricci tensor ${\rm Ric}_f$ is bounded from below and $f$ or $|\nabla f|$ is bounded. Later, Munteanu-Wang derived a gradient estimate for positive $f$-harmonic functions and obtain  the strong Liouville property on complete non-compact smooth metric measure space $\left(M, g, e^{-f} {\rm Vol}_{g} \right)$  with non-negative $\infty$-Bakry-Emery Ricci curvature ${\rm Ric}_f$ under the conditions that $f$ is of linear growth, $|f|(x) \leq a r(x)+b$ in \cite{MuWa1}, where $a$ and $b$ are some positive constants, $r(x):=d(x_{0}, x)$ is the geodesic distance function to a fixed point $x_{0}$ in $M$.   Also see \cite{MuWa2}.
\end{rem}

\vskip 3mm

According to Lemma 3.1.1 in \cite{Shen1}, for any vector $y\in T_{x}M\setminus \{0\}$, $x\in M$, the covector $\xi =g_{y}(y, \cdot)\in T^{*}_{x}M$ satisfies
\be
F(x,y)=F^{*}(x, \xi)=\frac{\xi (y)}{F(x,y)}. \label{shenF311}
\ee
Conversely, for any covector $\xi \in T_{x}^{*}M\setminus \{0\}$, there exists a unique vector $y\in T_{x}M\setminus \{0\}$ such that $\xi =g_{y}(y, \cdot)\in T^{*}_{x}M$. Naturally, we define a map ${\cal L}: TM \rightarrow T^{*}M$ by
$$
{\cal L}(y):=\left\{
\begin{array}{ll}
g_{y}(y, \cdot), & y\neq 0, \\
0, & y=0.
\end{array} \right.
$$
It follows from (\ref{shenF311}) that
$$
F(x,y)=F^{*}(x, {\cal L}(y)).
$$
Thus ${\cal L}$ is a norm-preserving transformation. We call ${\cal L}$ the Legendre transformation on Finsler manifold $(M, F)$.

Let
$$
g^{*kl}(x,\xi):=\frac{1}{2}\left[F^{*2}\right]_{\xi _{k}\xi_{l}}(x,\xi).
$$
For any $\xi ={\cal L}(y)$, we have
\be
g^{*kl}(x,\xi)=g^{kl}(x,y), \label{Fdual}
\ee
where $\left(g^{kl}(x,y)\right)= \left(g_{kl}(x,y)\right)^{-1}$.

Given a smooth function $u$ on $M$, the differential $d u_x$ at any point $x \in M$,
$$
d u_x=\frac{\partial u}{\partial x^i}(x) d x^i
$$
is a linear function on $T_x M$. We define the gradient vector $\nabla u(x)$ of $u$ at $x \in M$ by $\nabla u(x):=\mathcal{L}^{-1}(d u(x)) \in T_x M$. In a local coordinate system, we can express $\nabla u$ as
\be \label{nabna}
\nabla u(x)= \begin{cases}g^{* i j}(x, d u) \frac{\partial u}{\partial x^i} \frac{\partial}{\partial x^j}, & x \in M_u, \\ 0, & x \in M \backslash M_u,\end{cases}
\ee
where $M_{u}:=\{x \in M \mid d u(x) \neq 0\}$ (\cite{Shen1}). In general, $\nabla u$ is only continuous on $M$, but smooth on $M_{u}$.

The Hessian of $u$ is defined by using Chern connection as
$$
\nabla^2 u(X, Y)=g_{\nabla u}\left(D_X^{\nabla u} \nabla u, Y\right).
$$
One can show that $\nabla^2 u(X, Y)$ is symmetric (\cite{Ohta3, WuXin}).

Let $W^{1, p}(M)(p \geq 1)$ be the space of functions $u \in L^p(M)$ with $\int_M[F(\nabla u)]^p d m+\int_M[\overleftarrow{F}(\overleftarrow{\nabla} u)]^p d m<\infty$ and $W_0^{1, p}(M)$ be the closure of $\mathcal{C}_0^{\infty}(M)$ under the (absolutely homogeneous) norm
\be
\|u\|_{W^{1, p}(M)}:=\|u\|_{L^p(M)}+\frac{1}{2}\|F(\nabla u)\|_{L^p(M)}+\frac{1}{2}\|\overleftarrow{F}(\overleftarrow{\nabla} u)\|_{L^p(M)},
\ee
where $\mathcal{C}_0^{\infty}(M)$ denotes the set of all smooth compactly supported functions on M and $\overleftarrow{\nabla} u$ is the gradient of $u$ with respect to the reverse metric $\overleftarrow{F}$. In fact $\overleftarrow{F}(\overleftarrow{\nabla} u)=F(\nabla(-u))$.

Now, we decompose the volume form $d m$ of $m$ as $d m=\mathrm{e}^{\Phi} d x^1 d x^2 \cdots d x^n$. Then the divergence of a differentiable vector field $V$ on $M$ is defined by
$$
\operatorname{div}_m V:=\frac{\partial V^i}{\partial x^i}+V^i \frac{\partial \Phi}{\partial x^i}, \quad V=V^i \frac{\partial}{\partial x^i}.
$$
One can also define $\operatorname{div}_m V$ in the weak form by following divergence formula
$$
\int_M \phi \operatorname{div}_m V d m=-\int_M d \phi(V) d m
$$
for all $\phi \in \mathcal{C}_0^{\infty}(M)$.  Then the Finsler Laplacian $\Delta u$ is defined by
\be
\Delta u:=\operatorname{div}_m(\nabla u). \label{Lap}
\ee
From (\ref{Lap}), Finsler Laplacian is a nonlinear elliptic differential operator of the second order.
Further, noticing that $\nabla u$ is weakly differentiable, the Finsler Laplacian should be understood in the weak sense, that is, for $u \in W^{1,2}(M)$, $\Delta u$ is defined by
\be
\int_M \phi \Delta u d m:=-\int_M d \phi(\nabla u) dm  \label{Lap1}
\ee
for $\phi \in \mathcal{C}_0^{\infty}(M)$ (\cite{Ohta1,Shen1}).

Given a weakly differentiable function $u$ and a vector field $V$ which does not vanish on $M_u$, the weighted Laplacian of $u$ on the weighted Riemannian manifold $\left(M, g_V, m\right)$ is defined by
\be
\Delta^{V} u:= {\rm div}_{m}\left(\nabla^V u\right),
\ee
where
$$
\nabla^V u:= \begin{cases}g^{ij}(x, V) \frac{\partial u}{\partial x^i} \frac{\partial}{\partial x^j} & \text { for } x \in M_u, \\ 0 & \text { for } x \notin M_u .\end{cases}
$$
Similarly, the weighted Laplacian should be understood in the weak sense. We note that $\nabla^{\nabla u}u=\nabla u$ and $\Delta^{\nabla u} u=$ $\Delta u$. Moreover, it is easy to see that $\Delta u= {\rm tr}_{\nabla u} \nabla^2 u-{\bf S}(\nabla u)$ on $M_u$ (\cite{Ohta1, WuXin}).

The following Bochner-Weitzenb\"{o}ck formulas on Finsler manifolds given by Ohta-Sturm  are important for our discussions.

\begin{thm} {\rm (Pointwise Bochner-Weitzenb\"{o}ck Formula  \cite{Ohta1,Ohta3})} \label{PBW}
For $u \in \mathcal{C}^\infty(M)$, we have
 \be
 \Delta^{\nabla u}\left(\frac{F^{2}(x,\nabla u)}{2}\right)-d(\Delta u)(\nabla u)=\mathrm{Ric}_{\infty}(\nabla u)+\left\|\nabla^{2}u\right\|_{{\rm HS}(\nabla u)}^{2} \label{PBW1}
 \ee
 on $M_u$. Here $\left\| \ \cdot \ \right\|_{{\rm HS}(\nabla u)}$ denotes the Hilbert-Schmidt norm with respect to $g_{\nabla u}$.
\end{thm}

\begin{thm}{\rm (Integrated Bochner-Weitzenb\"{o}ck Formula \cite{Ohta1,Ohta3})} \label{BW}
For $u \in W_{\mathrm{loc}}^{2,2}(M) \bigcap \mathcal{C}^1(M)$ with $\Delta u \in W_{\mathrm{loc}}^{1,2}(M)$, we have
\beq
&& -\int_M d \phi\left(\nabla^{\nabla u}\left[\frac{F^2(x, \nabla u)}{2}\right]\right) d m  \nonumber \\
&& \ \ =\int_M \phi \left\{d \left(\Delta u\right)(\nabla u)+{\rm Ric}_{\infty}(\nabla u)+\left\|\nabla^2 u\right\|_{H S(\nabla u)}^2\right\} dm \label{BWforinf}
\eeq
for all non-negative functions $\phi \in W_0^{1,2}(M) \bigcap L^{\infty}(M)$.
\end{thm}

\section{Volume comparison theorem}\label{Volcom}
Let $(M, F, m)$ be an $n$-dimensional Finsler manifold with a smooth measure $m$ and $x \in M$.
For a unit vector $v \in T_x M$, let $\rho(v):=\sup \left\{t>0 \mid \text{the geodesic} \ \exp _x(tv) \ \text{is minimal} \right\}$. If $\rho(v)<\infty$, we call $\exp _x\left(\rho(v) v\right)$ a cut point of $x$. The set of all the cut points of $x$ is said to be the cut locus of $x$, denoted by $Cut(x)$. The cut locus of $x$ always has null measure and $d_{x}:= d_{F} (x, \cdot)$ is $C^1$ outside the cut locus of $x$.
 Let $\mathcal{D}_x:=M \backslash(\{x\} \cup C u t(x))$ be the cut-domain on $M$. For any $z \in \mathcal{D}_x$, we can choose the geodesic polar coordinates $(r, \theta)$ centered at $x$ for $z$ such that $r(z)=F(x, v)$ and $\theta^{\alpha}(z)=\theta^{\alpha}\left(\frac{v}{F(x,v)}\right)$, where $r(z)=d_{F}(x, z)$  and $v=\exp _x^{-1}(z) \in T_x M \backslash\{0\}$. It is well known that the distance function $r$ starting from $x \in M$ is smooth on $\mathcal{D}_x$ and $F(z, \nabla r(z))=1$ (\cite{BaoChernShen,Shen1}).  A basic fact is that the distance function $r=d_{F}(x, \cdot)$ satisfies the following (\cite{Shen1, WuXin})
\be
\nabla r |_{z}= \frac{\pa}{\pa r}|_{z}.
\ee
By Gauss's lemma, the unit radial coordinate vector $\frac{\partial }{{\partial r}}$ and the coordinate vectors $\frac{\partial }{{\partial {\theta ^\alpha }}}$ for $1\leq \alpha \leq n-1$ are mutually vertical with respect to $g_{\nabla r}$ (Lemma 6.1.1 in \cite{BaoChernShen}).
Further, we can write the volume form at $z=\exp _{x}(r\xi)$  with $v=r \xi$ as $\left.dm\right|_{\exp _x(r \xi)}=\sigma(x, r, \theta) dr d\theta$, where $\xi \in I_{x}:=\left\{\xi \in T_x M \mid F(x, \xi)=1\right\}$.
Then, for forward geodesic ball $B_{R}=B_R^{+}(x)$ of radius $R$ at the center $x \in M$, the volume of $B_R$ is
$$
m(B_R)=\int_{B_R} d m=\int_{B_R \cap \mathcal{D}_x} d m=\int_0^{R} dr \int_{\mathcal{D}_x(r)} \sigma(x, r, \theta) d\theta,
$$
where $\mathcal{D}_x(r)=\left\{\xi \in I_x \mid r \xi \in \exp_x^{-1}\left(\mathcal{D}_x\right)\right\}$.
Obviously, for any $0<s<t<R$, $\mathcal{D}_x(t)\subseteq\mathcal{D}_x(s)$.
Besides, by the definition of Laplacian, we have(\cite{Shen1,WuXin})
\be
\Delta r=\frac{\partial}{\partial r}\log\sigma.\label{lap}
\ee

We first give the  following Laplacian comparison theorem for distance functions.

\begin{thm}\label{Laplace}
Let $(M,F,m)$ be an $n$-dimensional forward complete Finsler measure space. Assume that ${\rm Ric}_{\infty}\geq -K$ for some $K\geq 0$ and $|\tau|\leq ar+b$, where $a,b$ are some non-negative constants, $r=d_{F}(x_0,x)$ is the distance function. Then we have
$$
\Delta r \leq \left( n-1+4b\right) \frac{1}{r}+3a+\frac{K}{3}r.
$$
\end{thm}
\begin{proof}
Let $\gamma: [0, r]\rightarrow M$ be the minimizing geodesic from $\gamma(0)=x_0$ to $\gamma(r)=x$.
By using the Bochner-Weitzenb\"{o}ck formula (\ref{PBW1}), we have
\be
0=\Delta^{\nabla r}\left[\frac{F^2(\nabla r)}{2}\right]= {\rm Ric}_{\infty}(\nabla r)+d(\Delta r)(\nabla r)+\|\nabla^{2}r\|_{{\rm HS}(\nabla r)}^2. \label{1}
\ee
Because $\nabla r$ is a geodesic field of $F$, we have
$$
\nabla^{2}r(\nabla r, V)=g_{\nabla r}\left(D_{\nabla r}^{\nabla r}\nabla r, V\right)=0
$$
for any $V \in T_{x}M$. Hence, we can obtain the following
\be
||\nabla^2r||_{{\rm HS}(\nabla r)}^2\ge\frac{1}{n-1}\left(tr_{\nabla r}\left(\nabla^2r\right)\right)^2=\frac{1}{n-1}(\Delta r+\mathbf{S}(\nabla r))^2,\label{2}
\ee
where $tr_{\nabla r}\left(\nabla^2r\right)$ denotes the trace of $\nabla^2r$ with respect to $g_{\nabla r}$. Thus, by the assumption that $\mathrm{Ric}_{\infty}\geq-K$, and combining (\ref{1}) and (\ref{2}), we get
\[
\frac{\partial(\Delta r)}{\partial r}\leq -\frac{1}{n-1}\left(\Delta r+\mathbf{S}(\nabla r)\right)^{2}+K.
\]
Then, we consider the formula
\begin{eqnarray*}
\frac{\partial(r^{2}\Delta r)}{\partial r}& =&2r\Delta r+r^{2}\frac{\partial(\Delta r)}{\partial r} \\
&\leq& 2r\Delta r-\frac{r^{2}}{n-1}(\Delta r+\mathbf{S}(\nabla r))^{2}+Kr^{2} \\
&=& -\left(\frac{r(\Delta r+\mathbf{S}(\nabla r))}{\sqrt{n-1}}-\sqrt{n-1}\right)^{2}+(n-1)-2r\mathbf{S}(\nabla r)+Kr^{2} \\
&\leq& (n-1)-2r\mathbf{S}(\nabla r)+Kr^{2}.
\end{eqnarray*}
Integrating it along $\gamma$ from $0$ to $r$ yields
\begin{eqnarray*}
r^{2}\Delta r &\leq& (n-1)\int_{0}^{r}dt-2\int_{0}^{r}t \ \mathbf{S}(\gamma(t),\dot{\gamma}(t))dt+\int_{0}^{r}Kt^{2}dt\\
&\leq &(n-1)r+2r(ar+b)+\int_{0}^{r}2(at+b)dt+\frac{K}{3}r^{3} \\
&=& (n-1+4b)r+3ar^{2}+\frac{K}{3}r^{3}.
\end{eqnarray*}
where we have used (\ref{S}) and $|\tau|\leq ar+b$ in the second inequality. It means that
$$\Delta r\leq(n-1+4b)\frac{1}{r}+3a+\frac{K}{3}r.$$
This completes the proof.
\end{proof}

From Theorem \ref{Laplace}, we get the following volume comparison theorem.

\begin{thm}\label{Volume}
Let $(M,F,m)$ be an $n$-dimensional forward complete Finsler measure space. Assume that ${\rm Ric}_{\infty}\geq -K$ for some $K\geq 0$ and $|\tau|\leq ar+b$, where $a,b$ are some non-negative constants, $r=d_{F}(x_0,x)$ is the distance function. Then, along any minimizing geodesic starting from the center $x_0$ of $B_R(x_0)$, we have the following for any $0<r_1<r_2<R$,
\be
\frac{m(B_{r_2}(x_0))}{m(B_{r_1}(x_0))}\leq \left(\frac{r_2}{r_1}\right)^{n+4b}e^{3ar_{2}+\frac{K}{6}r^{2}_{2}}. \label{volume}
\ee
\end{thm}
\begin{proof}
Let $\gamma:[0,r]\rightarrow M$ be the minimizing geodesic from $\gamma(0)=x_0$ to $\gamma(r)=x$. By using the geodesic polar coordinates $(r,\theta)$ centered at $x_0$ for $x$ and by (\ref{lap}), the Laplacian of the distance function $r$ satisfies
$$\Delta r=\frac{\partial}{\partial r}\log\sigma(x_0,r,\theta).$$
By Theorem \ref{Laplace}, we have
$$
\frac{\partial}{\partial r}\log\sigma(x_0,r,\theta)\leq(n-1+4b)\frac{1}{r}+3a+\frac{K}{3}r.
$$
For $0<r_1<r_2<R$, integating this from $r_1$ to $r_2$ yields
$$\int_{r_1}^{r_2}\frac{\partial}{\partial r}\log\sigma(x_0,r,\theta)dr\leq\int_{r_1}^{r_2}\left(\Big(n-1+4b\Big)\frac{1}{r}+3a+\frac{K}{3}r\right)dr,$$
which implies that
$$\frac{\sigma\left(x_0,r_2,\theta\right)}{\sigma\left(x_0,r_1,\theta\right)}\leq\left(\frac{r_2}{r_1}\right)^{n-1+4b}e^{3a\left(r_2-r_1\right)+\frac{K}{6}\left(r_2^2-r_1^2\right)}.$$
Further, for any $0<s<r_1<t<r_2<R$, we have
\beqn
\sigma\left(x_0,t,\theta\right)s^{n-1+4b}& \leq & t^{n-1+4b}\sigma\left(x_0,s,\theta\right)e^{3a(t-s)+\frac{K}{6}\left(t^{2}-s^{2}\right)} \\
&\leq& t^{n-1+4b}\sigma\left(x_0,s,\theta\right)e^{3ar_{2}+\frac{K}{6}r_{2}^{2}}.
\eeqn
Now, integrating it in $t$ from $r_1$ to $r_2$, we get
$$
s^{n-1+4b}\int_{r_1}^{r_2}\sigma\left(x_0,t,\theta\right)dt\leq\sigma\left(x_0,s,\theta\right)\frac{r_2^{n+4b}-r_1^{n+4b}}{n+4b}e^{3ar_2+\frac{K}{6}r_2^2}.
$$
Then, integrating it on both sides of above inequality with respect to $s$ from $0$ to $r_1$ yields
$$\frac{r_{1}^{n+4b}}{n+4b}\int_{r_{1}}^{r_{2}}\sigma\left(x_0,t,\theta\right)dt\leq\frac{r_{2}^{n+4b}-r_{1}^{n+4b}}{n+4b}e^{3ar_{2}+\frac{K}{6}r_{2}^{2}}\int_{0}^{r_{1}}\sigma\left(x_0,s,\theta\right)ds,$$
from which we obtain the following
$$\int_{r_1}^{r_2}dt\int_{\mathcal{D}_{x_0}(t)}\sigma(x_0,t,\theta)d\theta\leq\left[\left(\frac{r_2}{r_1}\right)^{n+4b}-1\right]e^{3ar_2+\frac{K}{6}r_2^2}\int_0^{r_1}ds\int_{\mathcal{D}_{x_0}(s)}\sigma(x_0,s,\theta)d\theta,$$
which means
$$m\left(B_{r_2}(x_0)\right)-m\left(B_{r_1}(x_0)\right)\leq\left[\left(\frac{r_2}{r_1}\right)^{n+4b}-1\right]e^{3ar_2+\frac{K}{6}r_2^2}m\left(B_{r_1}(x_0)\right).$$
Therefore, it follows that
$$\frac{m\left(B_{r_2}(x_0)\right)}{m\left(B_{r_1}(x_0)\right)}\leq\left(\frac{r_2}{r_1}\right)^{n+4b}e^{3ar_2+\frac{K}{6}r_2^2}.$$
This finishes the proof of Theorem \ref{Volume}.
\end{proof}

\begin{rem}
By (\ref{volume}), we can get the following volume comparison
\be
\frac{m(B_{r_2}(x_0))}{m(B_{r_1}(x_0))}\leq \left(\frac{r_2}{r_1}\right)^{n+4b}e^{3aR+\frac{K}{6}R^{2}}, \label{doubvol}
\ee
which implies the  volume doubling property of $(M, F, m)$, that is, there is a uniform constant $D_{0}$ such that $m(B_{2r}(x_0))\leq D_{0} m(B_{r}(x_0))$ for any $x_0 \in M$ and $0< r < R/2$.
\end{rem}

\section{Harnack inequality}\label{sharnack}

In this section, we will give the $p$-Harnack inequality for harmonic functions. For this aim, we first show  a $L^p$ mean value inequality for non-negative subsolutions of a class of elliptic equations.

As we know, C. Xia proved a local uniform Poincar\'{e} inequality under the condition that ${\rm Ric}_{N}\geq -K$ for $N \in [n,\infty)$ and $K > 0$ in \cite{XiaC}. Cheng-Feng got a local uniform Poincar\'{e} inequality under the conditions that ${\rm Ric}_{\infty} \geq K$ and $|\tau|\leq k$ for some $K \in \mathbb{R} $ and $k>0$ in \cite{CF}.  By the similar arguments, we can get the following local uniform Poincar\'{e} inequality under the conditions that ${\rm Ric}_{\infty}\geq -K$ and $|\tau|\leq ar+b$ for some non-negative constants $K$ and  $a,b$.

\begin{prop}\label{Poincare}
Let $(M,F,m)$ be an $n$-dimensional forward complete Finsler measure space with finite reversibility $\Lambda$. Assume that ${\rm Ric}_{\infty}\geq -K$ for some $K\geq 0$ and $|\tau|\leq ar+b$, where $a,b$ are some non-negative constants, $r=d(x_0,x)$ is the distance function. Then, there exist positive constants $c_1=c_1(n, b,\Lambda)$ and $c_2=c_2(\Lambda)$ depending only on $n$, $b$ and the reversibility $\Lambda$ of $F$, such that
\be
\int_{B_R}|u-\bar{u}|^2 dm\leq c_1e^{c_2(aR+KR^2)}R^2\int_{B_R}F^{*2}(du)dm\label{poin2}
\ee
for $u\in W^{1,2}_{\rm loc}(M)$, where $\bar{u}:=\frac{1}{m(B_R)}\int_{B_R} u\ dm$.
\end{prop}

As Saloff-Coste mentioned that a Poincar\'{e} inequality and the volume doubling property of the measure imply a family of local Sobolev inequalities on Riemannian manifolds (\cite{SC}).  Based on the Theorem \ref{Volume} and  Poincar\'{e} inequality (\ref{poin2}),  we can drive the following local uniform Sobolev inequality. The proof mainly follows the arguments of Theorem 4.3 in \cite{CF}.

\begin{prop}\label{Sobolev}
Let $(M,F,m)$ be an $n$-dimensional forward complete Finsler measure space with finite reversibility $\Lambda$. Assume that ${\rm Ric}_{\infty}\geq -K$ for some $K\geq 0$ and $|\tau|\leq ar+b$, where $a,b$ are some non-negative constants, $r=d_{F}(x_{0},x)$ is the distance function. Then, there exist positive constants $\nu (n,b)>2$ and $c=c(n, b,\Lambda)$ depending only on $n$, $b$ and the reversibility $\Lambda$ of $F$, such that
\be
\left(\int_{B_R}|u-\bar{u}|^{\frac{2\nu}{\nu-2}}dm\right)^{\frac{\nu-2}{\nu}}\leq e^{c(1+aR+KR^2)}R^2m(B_R)^{-\frac{2}{\nu}}\int_{B_R}F^{*2}(du)dm \label{sob1}
\ee
for $u\in W^{1,2}_{\rm loc}(M)$, where  $\nu=4(n+4b)-2$. Further,
\be
\left(\int_{B_R}|u|^{\frac{2\nu}{\nu-2}}dm\right)^{\frac{\nu-2}{\nu}}\leq e^{c(1+aR+KR^2)}R^2m(B_R)^{-\frac{2}{\nu}}\int_{B_R}(F^{*2}(du)+R^{-2}u^2)dm. \label{sob2}
\ee
\end{prop}

In the following, we will show how a localized Sobolev inequality implies certain $L^{p}$ mean value inequalities for non-negative subsolutions of a class of elliptic equations by using Moser iteration.

\begin{thm}\label{meanineq}
Let $(M, F, m)$ be an $n$-dimensional forward complete Finsler measure space with finite reversibility $\Lambda$. Assume that ${\rm Ric}_{\infty}\geq -K$ for some $K\geq 0$ and $|\tau|\leq ar+b$, where $a,b$ are some non-negative constants, $r=d(x_0,x)$ is the distance function. Suppose that $u$ is a non-negative function defined on $B_R$ satisfying
$$
\Delta u \geq -fu
$$
in the weak sense, where $f\in L^{\infty}(B_R)$ is non-negative. Then for any $p \in(0,\infty)$ and $\delta \in(0,1)$, there are constants $\nu=4(n+4b)-2 >2$ and $C=C(n, b, p, \Lambda)>0$ depending on $n, b, p$ and  $\Lambda$, such that
\begin{equation}\label{meanineq-1}
\sup _{B_{\delta R}} u^p \leq e^{C(1+aR+KR^2)}(1+\mathcal{A}R^2)^{\frac{\nu}{2}} (1-\delta)^{-\nu} m\left(B_R\right)^{-1}  \int_{B_R} u^{p} dm,
\end{equation}
where $\mathcal{A}:=\sup\limits_{B_R}|f|$.
\end{thm}

\begin{proof}
Since $u$ is a non-negative function satisfying $\Delta u\geq-fu$ in the weak sense on $B_R$, we have
\begin{equation}{\label{mean-2}}
\int_{B_{R}} d \phi(\nabla u) d m \leq \int_{B_{R}}\phi fu dm
\end{equation}
for any non-negative function $\phi \in \mathcal{C}_0^{\infty}\left(B_R\right)$.  For any $0 < \delta<\delta^{\prime} \leq 1$ and $s\geq1$, let $\phi=u^{2s-1}\varphi^2$,  where $\varphi$ is a cut-off function defined by
\begin{equation}\label{cut}
\varphi(x)=
\begin{cases}
1 & \text { on } B_{\delta R}, \\ \frac{\delta^{\prime} R- d_{F} \left(x_0, x\right)}{\left(\delta^{\prime}-\delta\right) R} & \text { on } B_{\delta^{\prime} R} \backslash B_{\delta R}, \\ 0 & \text { on } M \backslash B_{\delta^{\prime} R}.
\end{cases}
\end{equation}
Then $F^*(-d \varphi) \leq \frac{1}{\left(\delta^{\prime}-\delta\right) R}$ and hence $F^*(d \varphi) \leq \frac{\Lambda}{\left(\delta^{\prime}-\delta\right) R}$ a.e. on $B_{\delta^{\prime} R}$. Thus, by (\ref{mean-2}), we have
\begin{equation}\label{mean-4}
(2s-1)\int_{B_R} \varphi^2 u^{2s-2}F^{*2}(d u) d m  +2 \int_{B_R} \varphi u^{2s-1} d \varphi(\nabla u ) d m \leq  \int_{B_R} \varphi^2u^{2s} f d m.
\end{equation}
Therefore,
\[
s^2\int_{B_R} \varphi^2 u^{2s-2}F^{*2}(d u) d m \leq -2s \int_{B_R} \varphi u^{2s-1} d \varphi(\nabla u ) d m + s \int_{B_R} \varphi^2 u^{2s} f d m,
\]
that is,
\beqn
\int_{B_R} \varphi^2 F^{*2}(d u^{s}) d m &\leq & -2 \int_{B_R} \varphi u^{s} d \varphi(\nabla u^{s} ) d m + s \int_{B_R} \varphi^2 u^{2s} f d m \\
&\leq & 2 \int_{B_R} \varphi u^{s} F^{*}(-d \varphi) F\left(\nabla u^{s}\right)dm + s\int_{B_R}\varphi^2 u^{2s} f d m \\
& \leq & \frac{1}{2} \int_{B_R} \varphi^2 F^2(\nabla u^{s}) dm + 2\int_{B_R}u^{2s} F^{*2}(-d \varphi) dm + s \int_{B_R} \varphi^2 u^{2s} f d m.\\
\eeqn
Then we have the following
\beq
\int_{B_R} \varphi^2 F^{*2}(d u^s) dm &\leq &  4 \int_{B_R} u^{2s} F^{* 2}(-d \varphi) d m + 2s \int_{B_R} \varphi^2 u^{2s} f d m \nonumber \\
&\leq & \frac{4}{\left(\delta^{\prime}-\delta\right)^2 R^2}\int_{B_{\delta^{'}R}} u^{2s} dm + 2s \mathcal{A}\int_{B_{\delta^{'}R}} u^{2s} d m. \label{meanv1}
\eeq
Furthermore, it is easy to see that
\be
\int_{B_R} \varphi^{2s} F^{*2}(d u^s) dm \leq \int_{B_R} \varphi^{2} F^{*2}(d u^s) dm \label{meanv2}
\ee
and
\be
\int_{B_R} u^{2s} F^{*2}(d \varphi^s) dm =s^2\int_{B_R} u^{2s} \varphi^{2s-2} F^{*2}(d \varphi) dm \label{meanv3}
\leq  \frac{s^2\Lambda^2}{(\delta'-\delta)^2R^2} \int_{B_{\delta'R}} u^{2s} dm.
\ee
From (\ref{meanv1}), (\ref{meanv2}), (\ref{meanv3}) and  by H\"{o}lder's inequality and Sobolev inequality (\ref{sob2}), we obtain
\beq
&& \int_{B_{\delta R}} u^{2s\left(1+\frac{2}{\nu}\right)}d m \leq\int_{B_R}(u \varphi)^{2s\left(1+\frac{2}{\nu}\right)} d m \leq\left(\int_{B_R}(u^s \varphi^s)^{\frac{2 \nu}{\nu-2}} d m\right)^{\frac{\nu-2}{\nu}} \cdot\left(\int_{B_R}(u^s \varphi^s)^2 d m\right)^{\frac{2}{\nu}} \nonumber \\
&& \leq \mathcal{B} \int_{B_R}\left(F^{*2}(d(u^s \varphi^s))+R^{-2} u^{2s} \varphi^{2s}\right) dm \cdot\left(\int_{B_R} u^{2s} \varphi^{2s} d m\right)^{\frac{2}{\nu}} \nonumber\\
&& \leq \mathcal{B} \int_{B_R}\left(2 \varphi^{2s} F^{* 2}(du^s)+2 u^{2s} F^{* 2}(d \varphi^s)+R^{-2} u^{2s} \varphi^{2s}\right) d m \cdot\left(\int_{B_{\delta^{\prime} R}} u^{2s} d m\right)^{\frac{2}{\nu}} \nonumber\\
&& \leq \mathcal{B}\left(\frac{8+2s^2\Lambda^2}{\left(\delta^{\prime}-\delta\right)^2 R^2}+\frac{1}{R^2}+4s\mathcal{A}\right)\left(\int_{B_{\delta^{\prime} R}} u^{2s} d m\right)^{1+\frac{2}{\nu}} \nonumber\\
&& \leq 11s^2\Lambda^2\mathcal{B}\left(\frac{1+\mathcal{A}R^2}{\left(\delta^{\prime}-\delta\right)^2 R^2}\right)\left(\int_{B_{\delta^{\prime} R}} u^{2s} d m\right)^{1+\frac{2}{\nu}}, \label{meanv4}
\eeq
where $\mathcal{B}:=e^{c(1+aR+KR^2)}R^2m(B_R)^{-\frac{2}{\nu}}$, $\nu$ and $c$ were chosen as in Proposition \ref{Sobolev},  and we have used the fact that $F^{*2}(du+d\varphi)\leq 2F^{*2}(du)+2F^{*2}(d\varphi)$ in the third line.

Let $t:=1+\frac{2}{\nu}$.  The inequality (\ref{meanv4}) implies that
\begin{equation}
\int_{B_{\delta R}} u^{2 s t } d m \leq \frac{\mathcal{B}\Xi s^2}{\left(\delta^{\prime}-\delta\right)^2 R^2}\left(\int_{B_{\delta^{\prime} R}} u^{2 s } d m\right)^t, \label{mean-3}
\end{equation}
where $\Xi:=11\Lambda^2(1+\mathcal{A}R^2)$.

Now we consider the case when  $p\geq 2$. For $s\geq 1$, we choose $k\geq 1$ such that $s=\frac{p}{2}k$. Then (\ref{mean-3}) is rewritten as
\be
\int_{B_{\delta R}} u^{pkt}dm\leq \frac{\mathcal{B}\Xi p^{2}k^{2}}{4(\delta'-\delta)^2R^2}\left(\int_{B_{\delta' R}}u^{pk}dm\right)^t. \label{m1}
\ee
For any $0<\delta <1$, let $\delta_0=1$ and $\delta_{i+1}=\delta_i-\frac{1-\delta}{2^{i+1}}$ on $B_{\delta_i R}$, $i=0,1, \cdots$. Applying (\ref{m1}) for $\delta^{\prime}=\delta_i, \delta=\delta_{i+1}$ and $k=t^i$, we have
$$
\int_{B_{\delta_{i+1}R}}u^{pt^{i+1}}dm\leq \frac{4^i\mathcal{B}\Xi p^2t^{2i}}{(1-\delta)^2R^2}\left(\int_{B_{\delta_{i}R}}u^{pt^i}dm\right)^t.
$$
By iteration, we can get the following
\beqn
||u^p||_{L^{t^{i+1}}(B_{\delta_{i+1}R})}&=& \left(\int_{B_{\delta_{i+1}R}}u^{pt^{i+1}}dm\right)^{\frac{1}{t^{i+1}}}\\
&\leq& 4^{\sum jt^{-j}}\left[\frac{\mathcal{B}\Xi p^2}{(1-\delta)^2R^2}\right]^{\sum t^{-j}}t^{2\sum jt^{-j}}\int_{B_R}u^p dm,
\eeqn
in which $\sum$ denotes the summation on $j$ from 1 to $i+1$. Since $\sum_{j=1}^{\infty} t^{-j}=\frac{\nu}{2}$ and $\sum_{j=1}^{\infty} j t^{-j}=\frac{\nu^2+2\nu}{4}$, we have
\beq
||u^p||_{L^{\infty}(B_{\delta R})}&\leq& (2t)^{\frac{\nu^2+2\nu}{2}}(\mathcal{B}\Xi)^{\frac{\nu}{2}}\left(\frac{p}{(1-\delta)R}\right)^{\nu}\int_{B_R}u^p dm\nonumber\\
&=& e^{C(1+aR+KR^2)}(1-\delta)^{-\nu}m(B_R)^{-1}(1+\mathcal{A}R^2)^{\frac{\nu}{2}}\int_{B_R}u^p dm,
\eeq
which implies (\ref{meanineq-1}) with $p\geq 2$, where $C:=\log \left[(2t)^{\frac{\nu^2+2\nu}{2}}(11\Lambda^2)^{\frac{\nu}{2}}p^\nu\right]+\frac{\nu}{2}c$. This is the complete proof of (\ref{meanineq-1}) when $p\geq 2$.

In the following, we consider the case when $0<p<2$. For any $0<\delta<1$, we choose a constant $\varepsilon\in (0,1)$ satisfying $0<\delta+\varepsilon\leq 1$. Let $\delta_0=\delta+\varepsilon$, $\delta_{i+1}=\delta_i-\frac{\varepsilon}{2^{i+1}}$ on $B_{\delta_iR}$, $i=0, 1, \cdots$. Then, by the similar argument as above, we can obtain
\beq
\sup\limits_{B_{\delta R}} u^2
&\leq& e^{C(1+aR+KR^2)}\varepsilon^{-\nu}m(B_R)^{-1}(1+\mathcal{A}R^2)^{\frac{\nu}{2}}\int_{B_{(\delta+\varepsilon)R}}u^2 dm\nonumber\\
&\leq& e^{C(1+aR+KR^2)}\varepsilon^{-\nu}m(B_R)^{-1}(1+\mathcal{A}R^2)^{\frac{\nu}{2}}\left(\sup\limits_{B_{(\delta+\varepsilon)R}} u^2\right)^{1-\frac{p}{2}}\int_{B_R}u^p dm \label{m2}
\eeq
Now, let $\lambda=1-\frac{p}{2}>0$ and $\mathcal{O}(\delta):=\sup\limits_{B_{\delta R}}u^2$. Choose $\delta_0=\delta$ and $\delta_i=\delta_{i-1}+\frac{1-\delta}{2^i}$ for $i=1, 2, \cdots$. Applying (\ref{m2}) for $\delta=\delta_{i-1}$ and $\delta+\varepsilon=\delta_i$, we have
$$\mathcal{O}(\delta_{i-1})\leq \tilde{\mathcal{B}}2^{i\nu}(1-\delta)^{-\nu}\mathcal{O}(\delta_i)^\lambda,$$
where $\tilde{\mathcal{B}}:=e^{C(1+aR+KR^2)}m(B_R)^{-1}(1+\mathcal{A}R^2)^{\frac{\nu}{2}}\int_{B_R}u^p dm$. By iterating, we get
\be
\mathcal{O}(\delta_0)\leq \mathcal{\tilde{B}}^{\sum \lambda^{j-1}} 2^{\nu\sum j\lambda^{j-1}}(1-\delta)^{-\nu\sum \lambda^{j-1}}\mathcal{O}(\delta_{i})^{\lambda^i}, \label{iter0}
\ee
in which $\sum$ denotes the summation on $j$ from $1$ to $i$. Obviously, we have $\lim\limits_{i\rightarrow \infty}\delta_i=1$ and $\lim\limits_{i\rightarrow \infty}\lambda^i=1$. Moveover,  $\sum\limits^{\infty}_{j=1}\lambda^{j-1}=\frac{2}{p}$ and $\sum\limits^{\infty}_{j=1}j\lambda^{j-1}$ converges.
Letting $i\rightarrow \infty$ on the both sides of (\ref{iter0}), there exists a positive constant $\tilde{C}=\tilde{C}(n, b, p, \Lambda)$  such that
$$\sup\limits_{B_{\delta R}}u^p \leq \mathcal{O}(\delta)^{\frac{p}{2}} \leq e^{\tilde{C}(1+aR+KR^2)}(1+\mathcal{A}R^2)^{\frac{\nu}{2}}(1-\delta)^{-\nu}m(B_{R})^{-1}\int_{B_{R}} u^p dm.$$
This completes the proof.
\end{proof}

By the similar discussions, we can obtain the mean value inequality for positive superharmonic functions.

\begin{thm}\label{supu-1}
Let $(M, F, m)$ be an $n$-dimensional forward complete Finsler measure space with finite reversibility $\Lambda$. Assume that ${\rm Ric}_{\infty} \geq -K$ for some $K\geq 0$ and $|\tau|\leq ar+b$, where $a,b$ are some non-negative constants, $r=d(x_0,x)$ is the distance function. Suppose that $u$ is a positive superharmonic function satisfying $\Delta u\leq 0$  on $B_R$. Then, for any $p\in (0,\infty)$ and $\delta \in (0,1)$,  there exists a positive constant $C=C\left(n, b, p, \Lambda\right)$ depending on $n, b, p$ and  $\Lambda$, such that
\be
\sup_{B_{\delta R}} u^{-p} \leq e^{C(1+aR+KR^2)} (1-\delta)^{-\nu} m\left(B_R\right)^{-1}  \int_{B_R} u^{-p} dm. \label{-mean}
\ee
\end{thm}

\vskip 3mm

In order to prove the $p$-Harnack inequality, the following lemma is necessary.

\begin{lem}{\rm(\cite{SC})}\label{measure} Suppose that $\{U_{\sigma} \mid 0< \sigma \leq 1\}$ is a family of measurable subsets of a measurable set $U \subset \mathbb{R}^{n}$ with the property that $U_{\sigma'}\subset U_{\sigma}$ if $\sigma'\leq\sigma$. Fix $0<\delta <1$. Let $\gamma$ and $C$ be positive constants and $0<\alpha_0\leq\infty$. Let $g$ be a positive measurable function defined on $U_1=U$ which satisfies
\begin{equation}\label{alpha}
\left(\int_{U_{\sigma'}}g^{\alpha_0}dm\right)^{\frac{1}{\alpha_{0}}}\leq \left[C(\sigma-\sigma')^{-\gamma}m(U)^{-1}\right]^{\frac{1}{\alpha}-\frac{1}{\alpha_0}} \left(\int_{U_{\sigma}}g^{\alpha}dm\right)^{\frac{1}{\alpha}}
\end{equation}
for all $\sigma$, $\sigma'$, $\alpha$ satisfying $0<\delta\leq\sigma'<\sigma\leq1$ and $0<\alpha\leq \min\{1, \frac{\alpha_0}{2}\}$. Assume further that $g$ satisfies
\be
m(\log g>\lambda)\leq Cm(U)\lambda^{-1} \label{Lcondi2}
\ee
for all $\lambda>0$. Then
\begin{equation}\label{U}
\left(\int_{U_{\delta}}g^{\alpha_0}dm\right)^{\frac{1}{\alpha_{0}}}\leq C_0m(U)^{\frac{1}{\alpha_0}},
\end{equation}
where $C_0$ depends only on $\delta$, $\gamma$, $C$ and a lower bound on $\alpha_0$.
\end{lem}

Now we are in the position to prove the $p$-Harnack inequality for positive harmonic functions.

\begin{proof}[Proof of Theorem \ref{Harnack}]
By $L^p$ mean value inequality in Theorem \ref{meanineq} with $f=0$, for $\rho\in(0,1)$ and $0<\delta<1$, we have
\be
\sup\limits_{B_{\rho \delta R}}u^p \leq e^{C(1+aR+KR^2)} (1-\rho)^{-\nu} m\left(B_{\delta R}\right)^{-1}  \int_{B_{\delta R}} u^p dm. \label{mean2-1}
\ee
Hence, in order to prove Theorem \ref{Harnack}, we just need to prove the following inequality
\begin{equation}
\int_{B_{\delta R}} u^p dm \leq e^{C_0(1+aR+KR^2)}m(B_{R})\inf\limits_{B_{\delta R}}u^p  \label{mean2-2}
\end{equation}
for any $0<\delta<1$, where $C_0:=C_0(n, \delta, b, p, \Lambda)>0$.

Since $u$ is a positive harmonic function, we have
$$
\int_{B_R} d\phi(\nabla u) dm=0
$$
for all nonnegative functions $\phi \in \mathcal{C}_0^{\infty}\left(B_R\right)$.   For any $0<\delta<\delta'\leq1$,  let $\phi=u^{-1}\varphi^2$, where $\varphi$ is a cut-off function defined by (\ref{cut}).
Thus,
\beqn
0 &=&  \int_{B_R}d(u^{-1}\varphi^2)(\nabla u) dm = -\int_{B_R}u^{-2}\varphi^2 F^2(\nabla u)dm+ 2\int_{B_R} u^{-1}\varphi d\varphi(\nabla u)dm \\
&=&-\int_{B_R}\varphi^2F^2(\nabla v)dm+2\int_{B_R}\varphi d\varphi(\nabla v)dm,
\eeqn
where $v=\log u$. Then
\beqn
\int_{B_R}\varphi^{2}F^{2}(\nabla v)dm &=& 2\int_{B_R}\varphi d\varphi(\nabla v)dm \leq 2\int_{B_R}\varphi F(\nabla v)F^{*}(d \varphi) dm  \\
&\leq & \frac{1}{2}\int_{B_R} \varphi^2F^2(\nabla v)dm+2\int_{B_R}F^{*2}(d\varphi)dm,
\eeqn
from which we can get
\be
\int_{B_{\delta R}}F^2(\nabla v)dm\leq4\int_{B_{\delta' R}}F^{*2} (d\varphi)dm\leq\frac{4\Lambda^2}{(\delta'-\delta)^2R^2}m({B_{\delta'R}}).  \label{h1}
\ee
In addition, it follows that
\beq
t \cdot m\left(B_{\delta R}\bigcap \{|v-\bar{v}|\geq t\}\right) & \leq &\int_{B_{\delta R}}|v-\bar{v}|dm \nonumber\\
&\leq &\left(\int_{B_{\delta R}}|v-\bar{v}|^2dm\right)^{\frac{1}{2}}m(B_{\delta R})^{\frac{1}{2}} \nonumber\\
&\leq &\,\sqrt{c_1}e^{\frac{1}{2}c_2(aR+KR^2)}R\left(\int_{B_{\delta R}}F^2(\nabla v)dm\right)^{\frac{1}{2}}m(B_{\delta R})^{\frac{1}{2}}\nonumber\\
& \leq &\,\sqrt{c_1}e^{\frac{1}{2}c_2(aR+KR^2)} \frac{2\Lambda}{\delta'-\delta}  m(B_{R}),  \label{h2}
\eeq
where $\bar{v}=\frac{1}{m\left(B_{\delta R}\right)} \int_{B_{\delta R}} v dm$ and we have used Poincar\'{e} inequality (\ref{poin2}) in the third inequality and the last inequality follows from (\ref{h1}).

In the following, we first consider the case when $p\geq 2$. Letting $g_1 =e^{-\bar{v}}u$ in Lemma \ref{measure}, we have $\Delta^{\nabla u} g_1=e^{-\bar{v}} \Delta u=0$. By the mean value inequality in Theorem \ref{meanineq}, we have
\beqn
\left(\int_{B_{\delta R}}g_1^p dm\right)^{\frac{1}{p}} &\leq & \left(\sup\limits_{B_{\delta R}}g_1^p \right)^{\frac{1}{p}} m(B_{\delta R})^{\frac{1}{p}}\leq   \left(\sup\limits_{B_{\delta R}}g_1^{\frac{1}{2}}\right)^2 m(B_{\delta R})^{\frac{1}{p}}\\
&\leq& e^{2C(1+aR+KR^2)}(1-\delta)^{-2\nu}m(B_{R})^{\frac{1-2p}{p}}\left(\int_{B_R}g_1^{\frac{1}{2}}dm\right)^2 ,
\eeqn
which means that (\ref{alpha}) holds for $\alpha_0=p$, $\alpha=\frac{1}{2}$ and $\sigma=1$. On the other hand, in this case, (\ref{h2}) means that (\ref{Lcondi2}) holds. Then (\ref{U}) becomes
\begin{equation}\label{e-v}
\int_{B_{\delta R}} u^p dm \leq e^{C_1(1+aR+KR^2)} m(B_{R})e^{p\bar{v}}.
\end{equation}
Similarly, letting $g_2=e^{\bar{v}}u^{-1}$ and taking $\alpha_0=p$, $\alpha=\frac{1}{2}$ and $\sigma=1$, by the similar argument, we can conclude from Theorem \ref{supu-1} and Lemma \ref{measure} that
\begin{equation}\label{e-u}
\sup\limits_{B_{\delta R}} \{u^{-p}\} \leq e^{C_2(1+aR+KR^2)}e^{-p\bar{v}}.
\end{equation}
Therefore, from $\inf\limits_{B_{\delta R}} \{u^p\}=\left(\sup\limits_{B_{\delta R}} \{u^{-p}\}\right)^{-1}$, we obtain the following
$$
\int_{B_{\delta R}} u^p dm \leq e^{C_3(1+aR+KR^2)} m(B_{R})\inf\limits_{B_{\delta R}}u^p,
$$
that is, (\ref{mean2-2}) holds. Then we immediately obtain by (\ref{mean2-1})
\begin{equation}
\sup\limits_{B_{\rho \delta R}}u^p \leq e^{C_4(1+aR+KR^2)} \frac{m(B_{R})}{m(B_{\delta R})}\inf\limits_{B_{\delta R}}u^p\leq e^{C_5(1+aR+KR^2)}\inf\limits_{B_{\rho\delta R}}u^p,
\end{equation}
which is the desired inequality by replacing $\rho\delta$ with $\delta$. This completes the proof in the case when $p\geq 2$.

Now we consider the case when $0<p<2$. Letting $\tilde{g}_1=e^{-\bar{v}}u$ and by the mean value inequality in Theorem \ref{meanineq}, we have
\beqn
\left(\int_{B_{\delta R}}\tilde{g}_1^p dm\right)^{\frac{1}{p}} &\leq & \left(\sup\limits_{B_{\delta R}}\tilde{g}_1^{\frac{p}{2}}\right)^{\frac{2}{p}} m(B_{\delta R})^{\frac{1}{p}}\\
&\leq& e^{\frac{2}{p}C(1+aR+KR^2)}(1-\delta)^{-\frac{2}{p}\nu}m(B_{R})^{-\frac{1}{p}}\left(\int_{B_R}\tilde{g}_1^{\frac{p}{2}}dm\right)^\frac{2}{p} ,
\eeqn
which means that (\ref{alpha}) holds for $\alpha_0=p$, $\alpha=\frac{p}{2}$ and $\sigma=1$. Further, by (\ref{h2}) again, we know that (\ref{Lcondi2}) holds for $\tilde{g}_1$. Then by (\ref{U}), we have
\begin{equation}\label{p<2}
\int_{B_{\delta R}} u^p dm \leq e^{C_6(1+aR+KR^2)} m(B_{R})^{-1}e^{p\bar{v}}.
\end{equation}
Moreover, letting $\tilde{g}_2=e^{\bar{v}}u^{-1}$ and taking $\alpha_0=p$, $\alpha=\frac{p}{2}$ and $\sigma=1$, we obtain
\begin{equation}\label{p<21}
{\left(\inf\limits_{B_{\delta R}} \{u^p\}\right)}^{-1}=\sup\limits_{B_{\delta R}} \{u^{-p}\} \leq e^{C_7(1+aR+KR^2)}m(B_R)e^{-p\bar{v}}.
\end{equation}
Then by the same argument as above, we can get (\ref{mean2-2}). This means that (\ref{Harnack}) holds.
\end{proof}

\vskip 2mm

By the same argument of Theorem \ref{Harnack}, we can get a weighted version of $p$-Harnack inequality, which is important for our discussions about the applications of the $p$-Harnack inequality.

\begin{thm}\label{Harnack2}
Let $(M, F, m)$ be an $n$-dimensional forward complete Finsler measure space equipped with a uniformly convex and uniformly smooth Finsler metric $F$. Assume that ${\rm Ric}_{\infty}\geq -K$ for some $K\geq 0$ and $|\tau|\leq ar+b$, where $a,b$ are some non-negative constants, $r=d(x_0,x)$ is the distance function. Let $V$ be a measurable  vector field on $M$ and  $v$ be a non-negative function such that $M_{v}\subset M_{V}=\{x \in M \mid V(x)\neq 0\}$. If $\Delta^{V} v =0$ on $B_{R}$, then for any $p\in (0,\infty)$ and $\delta \in (0,1)$, there exists a positive constant $C=C\left(n, \delta, b, p, \kappa, \kappa^*\right)$ depending on $n, \delta, b, p, \kappa$ and $\kappa^*$, such that
\begin{equation}
\sup\limits_{B_{\delta R}}v^p \leq e^{C(1+aR+KR^2)}\inf\limits_{B_{\delta R}}v^p.\label{harnackpv}
\end{equation}
\end{thm}

\begin{proof}
We divide the proof into two steps.
\vskip 2mm
\noindent{\bf Step 1}  We firstly demonstrate that a non-negative function $v$ satisfying $\Delta^{V} v\geq -fv$  for a non-negative function $f\in L^{\infty}(B_R)$ must satisfy a mean value inequality as follows,
\begin{equation}\label{mean-v}
\sup_{B_{\delta R}} v^p \leq e^{C(1+aR+KR^2)}(1+\mathcal{A}R^2)^{\frac{\nu}{2}} (1-\delta)^{-\nu} m\left(B_R\right)^{-1} \int_{B_R} v^{p} dm.
\end{equation}
The proof is similar to the argument of Theorem \ref{meanineq}.  Actually, starting from $\Delta^{V} v\geq -fv$, for any $s\geq 1$, we can obtain the following inequality,
\be
(2s-1)\int_{B_R} \varphi^2 v^{2s-2}dv(\nabla^{V}v) d m +2 \int_{B_R} \varphi v^{2s-1} d\varphi (\nabla^{V}v) d m \leq  \int_{B_R} \varphi^2v^{2s} f d m. \label{mean-4-2}
\ee
It is just an analogue of (\ref{mean-4}). From (\ref{mean-4-2}), we can obtain
\beq
& & s^{2}\kappa^{-1}\int_{B_R} \varphi^2 v^{2s-2}F^{*2}(d v) d m \leq -2s \int_{B_R} \varphi v^{2s-1} d\varphi (\nabla^{V}v) d m + s \int_{B_R} \varphi^2 v^{2s} f d m \nonumber\\
& & \leq  2s (\kappa^*)^{-1}\int_{B_R} \varphi v^{2s-1} F^*(-d\varphi)F^*(dv)d m + s \int_{B_R} \varphi^2 v^{2s} f d m, \label{mean-4-3}
\eeq
where we have used the fact that $dv(\nabla^{V}v)\geq \kappa^{-1}F^{*2}(dv)$ and $-d\varphi (\nabla^{V}v)\leq(\kappa^*)^{-1}F^*(-d\varphi)F^*(dv)$.
Further, by (\ref{mean-4-3}), we can get
\be
\int_{B_R}\varphi^2 F^{*2}(dv^s)dm\leq \frac{4\kappa^2}{\kappa^{*2}(\delta'-\delta)^2R^2}\int_{B_{\delta'R}}v^{2s}dm+2s\kappa \sup_{B_R}|f|\int_{B_{\delta'R}}v^{2s}dm. \label{meanv1-2}
\ee
It is just the analogue of (\ref{meanv1}).  Then,  by the same argument as the proof of Theorem \ref{meanineq}, we assert that (\ref{mean-v}) holds.
\vskip 2mm

\noindent{\bf Step 2} Now, we will prove the Harnack inequality (\ref{harnackpv}) by using (\ref{mean-v}). In this case,  $f=0$.  For $\rho\in (0,1)$ and $0<\delta<1$, we have
\beq\label{har}
\sup\limits_{B_{\rho \delta R}}v^p \leq e^{C(1+aR+KR^2)} (1-\rho)^{-\nu} m\left(B_{\delta R}\right)^{-1}  \int_{B_{\delta R}} v^p dm.
\eeq
In the following, we only need to prove that the following inequality holds,
\begin{equation}
\int_{B_{\delta R}} v^p dm \leq e^{C(1+aR+KR^2)}m(B_{R})\inf\limits_{B_{\delta R}}v^p. \label{mean-v1}
\end{equation}
The proof is similar to the argument of Theorem \ref{Harnack}. Since $v$ satisfies $\Delta^{V}v=0$, we have
$$
\int_{B_R} d\phi(\nabla^V v) dm=0
$$
for all non-negative functions $\phi \in \mathcal{C}_0^{\infty}\left(B_R\right)$. For any $0<\delta<\delta'\leq1$, let $\phi=v^{-1}\varphi^2$, where $\varphi$ is a cut-off function defined by (\ref{cut}).
Thus,
\beqn
0 = \int_{B_R}d(v^{-1}\varphi^2)(\nabla^{V}v) dm = -\int_{B_R}v^{-2}\varphi^2 dv(\nabla^{V}v)dm+ 2\int_{B_R} v^{-1}\varphi d\varphi(\nabla^{V}v)dm.
\eeqn
From it, we can obtain
\beq\label{kv-1}
\kappa^{-1}\int_{B_R} v^{-2}\varphi^2 F^{*2}(dv)dm\leq 2(\kappa^*)^{-1}\int_{B_R} v^{-1}\varphi F^*(d\varphi)F^*(dv)dm,
\eeq
where we have used the fact that $dv(\nabla^{V}v)\geq \kappa^{-1}F^{*2}(dv)$ and $d\varphi (\nabla^{V}v)\leq(\kappa^*)^{-1}F^*(d\varphi)F^*(dv)$.
Set $w:=\log v$, by (\ref{kv-1}), we can get
\beqn
\kappa^{-1}\int_{B_R}\varphi^{2}F^{*2}(d w)dm &=& 2(\kappa^*)^{-1}\int_{B_R}\varphi F^*(d w)F^{*}(d \varphi) dm  \\
&\leq & \frac{1}{2}\kappa^{-1}\int_{B_R} \varphi^2F^{*2}(d w)dm+\frac{2\kappa}{\kappa^{*2}}\int_{B_R}F^{*2}(d\varphi)dm,
\eeqn
from which we can get
$$\int_{B_{\delta R}}F^{*2}(dw)dm\leq\frac{4\kappa^2}{\kappa^{*2}}\int_{B_{\delta' R}}F^{*2} (d\varphi)dm\leq\frac{4\Lambda^2\kappa^2}{\kappa^{*2}(\delta'-\delta)^2R^2}m({B_{\delta'R}})\leq\frac{4\kappa^3}{\kappa^{*2}(\delta'-\delta)^2R^2}m({B_{\delta'R}}).$$
It is just the analogue of (\ref{h1}). Then, by the same argument as the proof of Theorem \ref{Harnack}, we can obtain (\ref{mean-v1}). Combining (\ref{har}) and (\ref{mean-v1}), we conclude that (\ref{harnackpv}) holds.
\end{proof}

\section{Some applications of Harnack inequality}\label{Appl}

In this section, we will show some applications of the Harnack inequality. Firstly, we give the H\"{o}lder continuity estimate for positive harmonic functions. As we know,
Moser gave a proof of the H\"{o}lder continuity estimate in \cite{M1} as an application of the elliptic Harnack inequality for positive solutions of uniformly elliptic equations in $\mathbb{R}^n$. In this following, we will give a H\"{o}lder continuity estimate on Finsler metric measure manifolds. We first provide the following lemma.

\begin{lem}\label{Holder1}
Let $(M, F, m)$ be an $n$-dimensional forward complete Finsler measure space equipped with a uniformly convex and uniformly smooth Finsler metric $F$. Assume that ${\rm Ric}_{\infty}\geq -K$ for some $K\geq 0$ and $|\tau|\leq ar+b$, where $a,b$ are some non-negative constants, $r=d(x_0,x)$ is the distance function. If $u$ is a positive harmonic function on $B_R$, then, for any $\rho \in (0,1)$, there exists a constant $\alpha=\alpha(n,a,b,\kappa,\kappa^*,R,K)$, such that
\begin{equation}
\sup\limits_{x_1, x_2 \in B_{\rho R}} \left\{|u(x_1)-u(x_2)|\right\} \leq 2^{\alpha} \rho^{\alpha} \sup\limits_{B_R}\{u\}.\label{lem1}
\end{equation}
\end{lem}

\begin{proof}
For $s\in(0,1]$, we define $\zeta(s):= \sup\limits_{B_{sR}}\{u\}$, $\epsilon(s):= \inf\limits_{B_{sR}}\{u\}$, $\omega(s):=\zeta(s)-\epsilon(s)$. Obviously, $\omega(s)$ is a monotonous non-decreasing function on $(0,1]$.

Since $u$ is a positive harmonic function on $B_{R}$, and $\nabla^{\nabla u}(\zeta(s)-u)=-\nabla^{\nabla u}u$, $\nabla^{\nabla u}(u-\epsilon(s))=\nabla^{\nabla u}u$, we have $\Delta^{\nabla u}(\zeta(s)-u)=0$ and $\Delta^{\nabla u}(u-\epsilon(s))=0$ on $B_{R}$.

For $s\in(0,1]$, applying Theorem \ref{Harnack2} for $V=\nabla u$, $p=1$ and $\delta=\frac{s}{2}$ to $\zeta(s)-u$ and $u-\epsilon(s)$, we get
$$
\sup\limits_{B_{\frac{s}{2}R}}\{\zeta(s)-u\}\leq e^{C(1+aR+KR^2)}\inf\limits_{B_{\frac{s}{2}R}}\{\zeta(s)-u\}
$$
and
$$
\sup\limits_{B_{\frac{s}{2}R}}\{u-\epsilon(s)\}\leq e^{C(1+aR+KR^2)}\inf\limits_{B_{\frac{s}{2}R}}\{u-\epsilon(s)\},
$$
which mean that
$$
\zeta(s)-\epsilon\left(\frac{s}{2}\right)\leq e^{C(1+aR+KR^2)} \left(\zeta(s)-\zeta\left(\frac{s}{2}\right)\right)
$$
and
$$
\zeta\left(\frac{s}{2}\right)-\epsilon(s)\leq e^{C(1+aR+KR^2)} \left(\epsilon\left(\frac{s}{2}\right)-\epsilon(s)\right).
$$
Hence, we obtain $$\zeta(s)-\epsilon(s)+\zeta\left(\frac{s}{2}\right)-\epsilon\left(\frac{s}{2}\right)\leq e^{C(1+aR+KR^2)}\left(\zeta(s)-\epsilon(s)-\zeta\left(\frac{s}{2}\right)+\epsilon\left(\frac{s}{2}\right)\right),$$
that is,
\be\label{**}
\zeta\left(\frac{s}{2}\right)-\epsilon\left(\frac{s}{2}\right)\leq \frac{e^{C(1+aR+KR^2)}-1}{e^{C(1+aR+KR^2)}+1}\left(\zeta(s)-\epsilon(s)\right).
\ee
Set $A:=\frac{e^{C(1+aR+KR^2)}-1}{e^{C(1+aR+KR^2)}+1}$.  It is obvious that $A\in(0,1)$, and $A=2^{-\alpha}$, where $\alpha:=\log_2 \frac{1}{A}$. Then, we can rewrite (\ref{**}) as
\beq\label{iter1}
\omega\left(\frac{s}{2}\right)\leq 2^{-\alpha}\omega(s).
\eeq
Choosing $s_i=\frac{1}{2^i}$, $i=0,1,2,\cdots$. Applying (\ref{iter1}) for $s=s_i$, we have $\omega(s_{i+1})\leq 2^{-\alpha}\omega(s_i)$, and then
$$
\omega(s_{i+1})\leq 2^{-\alpha}\omega(s_i)\leq \dots \leq 2^{-(i+1)\alpha}\omega(1).
$$

For $\rho\in(0,1)$, there exists a non-negative integer $k$, such that $\frac{1}{2^{k+1}}\leq \rho < \frac{1}{2^k}$. Then
$$
\omega(\rho)\leq \omega(2^{-k})<2^{-k\alpha}\omega(1)=2^{\alpha}2^{-(k+1)\alpha} \omega(1)\leq 2^{\alpha} \rho^\alpha \omega(1).
$$
By the above inequality, we obtain that
$$
\sup\limits_{x_1,x_2\in B_{\rho R}}\{|u(x_1)-u(x_2)|\}\leq 2^{\alpha}\rho^\alpha\sup\limits_{B_R}\{u\}.
$$
This completes the proof of Lemma \ref{Holder1}.
\end{proof}

\vskip 2mm

Based on Lemma \ref{Holder1}, we can give immediately the proof of Theorem \ref{Holder}.

\begin{proof}[Proof of Theorem \ref{Holder}]
Fix $\rho\in(0,1)$.  For $x_1,x_2\in B_{\rho R}$, We divide the proof into two cases.

\vskip 2mm

\noindent{\bf Case 1} $d_{F}(x_1,x_2)\geq \frac{(1-\rho)R}{\sqrt{\kappa}}$. In this case,
\beq
|u(x_1)-u(x_2)| \leq \sup\limits_{B_R}\{u\} &\leq& (1-\rho)^{-\alpha}\kappa^{\frac{\alpha}{2}}R^{-\alpha}d_{F}(x_{1}, x_{2})^{\alpha}\sup\limits_{B_R}\{u\} \nonumber\\
&\leq & 2^{\alpha} \left(\frac{\sqrt{\kappa} d_{F}(x_{1},x_{2})}{(1-\rho)R}\right)^{\alpha} \sup\limits_{B_R}\{u\}. \label{holder1}
\eeq

\vskip 1mm

\noindent{\bf Case 2} $d_{F}(x_{1},x_{2})< \frac{(1-\rho)R}{\sqrt{\kappa}}$. Let $\gamma: [0,d_{F}(x_{1},x_{2})]\rightarrow M$ be the unit speed minimal geodesic with $\gamma(0)=x_1$ and $\gamma(d_{F}(x_{1}, x_{2}))= x_{2}$.
Set $z_{0}:=\gamma\left(\frac{d_{F}(x_{1},x_{2})}{2}\right)$. Then we get
$$
d_{F}(x_{0},z_{0})\leq d_{F}(x_{0},x_{1})+d_{F}(x_{1},z_{0})< \rho R+\frac{(1-\rho)}{2\sqrt{\kappa}}R\leq \rho R+\frac{(1-\rho)}{2}R=\frac{(1+\rho)R}{2}.
$$
Hence, $z_{0}\in B_{\frac{(1+\rho)R}{2}}\subset B_R$.

Let $\widetilde{B}:= B^{+}_{\frac{(1-\rho)R}{2}}(z_0)$. For any $z\in \widetilde{B}$, $d_{F}(x_{0},z)\leq d_{F}(x_{0},z_{0})+d_{F}(z_{0},z)\leq \frac{(1+\rho)}{2}R+\frac{(1-\rho)}{2}R=R$. Thus, $\widetilde{B}\subset B_{R}$. Further,
because
$$
d_{F}(z_{0},x_{1})\leq \sqrt{\kappa} d_{F}(x_{1},z_{0})=\frac{\sqrt{\kappa}}{2}d_{F}(x_{1},x_{2})< \frac{(1-\rho)}{2}R
$$
and
$$
d_{F}(z_{0},x_{2})\leq \frac{1}{2}d_{F}(x_{1},x_{2})\leq \frac{\sqrt{\kappa}}{2}d_{F}(x_{1},x_{2})< \frac{(1-\rho)}{2}R,
$$
we have $x_{1},x_{2}\in B^{+}_{\frac{\sqrt{\kappa} d_{F}(x_{1},x_{2})}{2}}(z_0)\subset \widetilde{B}$. For $\frac{\sqrt{\kappa}d_{F}(x_{1},x_{2})}{(1-\rho)R}<1$,  applying Lemma \ref{Holder1} on $\widetilde{B}$, we can find that
\be
|u(x_1)-u(x_2)|\leq 2^{\alpha}\left(\frac{\sqrt{\kappa} d_{F}(x_{1},x_{2})}{(1-\rho)R}\right)^{\alpha} \sup\limits_{\widetilde{B}}\{u\}\leq 2^{\alpha}\left(\frac{\sqrt{\kappa} d_{F}(x_{1},x_{2})}{(1-\rho)R}\right)^{\alpha} \sup\limits_{B_R}\{u\}. \label{holder2}
\ee
Combining (\ref{holder1})  with (\ref{holder2}), we get (\ref{holder}). This completes the proof of  Theorem \ref{Holder}.
\end{proof}

\vskip 3mm

Next, we will study Liouville property for positive harmonic functions on a forward complete Finsler manifold satisfying ${\rm Ric}_{\infty}\geq 0$ and $|\tau|\leq b$ and give the proof of Theorem \ref{Liouville1} by using the Harnack inequality.

\begin{proof}[Proof of Theorem \ref{Liouville1}]
Let $u$ be a positive harmonic function on $M$ which is bounded below. Let $i(u):=\inf\limits_{M}\{u\}$ and $v:=u-i(u)$. It is easy to check that $\nabla v=\nabla u$ and $\Delta v=0$. Then, applying Theorem \ref{Harnack} on $B_{R}$ to $v$ and taking $\delta =\frac{1}{2}$, we have the following
$$
\sup\limits_{B_{\frac{1}{2}R}}\{u-i(u)\}\leq e^{C}\inf\limits_{B_{\frac{1}{2}R}}\{u-i(u)\}.
$$
Letting $R\rightarrow \infty$, we find that $\inf\limits_{B_{\frac{1}{2}R}}\{u-i(u)\}\rightarrow 0$. Then we conclude  that $u=i(u)$ is a constant on $M$.
\end{proof}

\vskip 3mm

Finally, we will give the local gradient estimate for positive harmonic functions and prove Theorem \ref{Grad}. As an important fundamental, we need the following lemma.
\begin{lem}\label{Mean}
Let $(M, F, m)$ be an $n$-dimensional forward complete Finsler measure space equipped with a uniformly convex and uniformly smooth Finsler metric $F$. Assume that ${\rm Ric}_{\infty}\geq -K$  and $|\tau|\leq ar+b$ for some non-negative constants $K$ and  $a,b$, $r=d_{F}(x_{0},x)$ is the distance function. Suppose that $u \in W^{2,2}(B_R)$ satisfies
\begin{equation}
\int_{B_R}\phi\Delta^{\nabla u} F^2(\nabla u)dm\geq -2 K \int_{B_R} \phi F^2(\nabla u) dm \label{bw}
\end{equation}
for all non-negative functions $\phi \in \mathcal{C}_0^{\infty}\left(B_R\right)$. Then for any $0<p<\infty$ and $0<\delta <1$, there are constants $\nu=4(n+4b)-2 >2$ and $C=C(n, b, \kappa, \kappa^*)$ such that
\begin{equation}
\sup_{B_{\delta R}} F(\nabla u)^{2p} \leq e^{C(1+aR+KR^2)}(1+KR^2)^{\frac{\nu}{2}} (1-\delta)^{-\nu} m\left(B_R\right)^{-1} \int_{B_R} F(\nabla u)^{2p} dm. \label{harnack}
\end{equation}
\end{lem}
\begin{proof}
Let $h(x):=F^2(\nabla u)$. Obviously,  $h\in W^{1,2}(B_R)$. For $0<\delta<\delta'\leq 1$, $s\geq 1$. Let $\phi:=h^{2s-1}\varphi^2$, where $\varphi$ is the cut-off function defined by (\ref{cut}).
By (\ref{bw}), we have
$$(2s-1)\int_{B_R}\varphi^2h^{2s-2}dh(\nabla^{\nabla u} h)dm+2\int_{B_R}\varphi h^{2s-1}d\varphi (\nabla^{\nabla u} h)dm\leq 2K\int_{B_R}\varphi^2 h^{2s}dm.$$
Then, by (\ref{F*sc}), we have $$s^2\kappa^{-1}\int_{B_R}\varphi^2h^{2s-2}F^{*2}(dh)dm\leq -2s\int_{B_R}\varphi h^{2s-1} d\varphi(\nabla^{\nabla u}h)dm+ 2sK \int_{B_R} \varphi^2h^{2s}dm,$$
which means that
\beqn
\kappa^{-1}\int_{B_R}\varphi^2F^{*2}(dh^s)dm &\leq& \frac{1}{2}\kappa^{-1}\int_{B_R}\varphi^2F^{*2}(dh^s)dm+\frac{2\kappa}{\kappa^{*2}}
\int_{B_R}h^{2s}F^{*2}(-d\varphi)dm\\
& &+2sK\int_{B_R}\varphi^2h^{2s}dm,
\eeqn
where we have used the fact that $-d\varphi(\nabla^{\nabla u} h)\leq (\kappa^*)^{-1}F^*(-d\varphi)F^*(dh)$.
By the above inequality, we get
\beqn
\int_{B_R}\varphi^2F^{*2}(dh^s)dm\leq\frac{4\kappa^2}{\kappa^{*2}(\delta'-\delta)^2R^2}\int_{B_{\delta'R}}h^{2s}dm+4s\kappa K\int_{B_{\delta'R}}h^{2s}dm,
\eeqn
which is just an analogue of (\ref{meanv1}). Then, by the similar argument in the proof of Theorem \ref{meanineq}, we can derive (\ref{harnack}).
\end{proof}

\vskip 2mm

Now, based on Lemma \ref{Mean}, we can  prove Theorem \ref{Grad}.

\begin{proof}[Proof of Theorem \ref{Grad}]
For any positive harmonic function $u$ satisfying $\Delta u=0$ in a weak sense on $M$, it is known that $u\in W^{2,2}_{loc}(M)\bigcap \mathcal{C}^{1,\alpha}(M)$ (\cite{Ohta1}). By Bochner-Weitzenb\"{o}ck formula (\ref{BWforinf}) and our assumption, we have
\begin{equation}
\int_{B_R}\phi\Delta^{\nabla u} F^2(\nabla u)dm\geq -2 K \int_{B_R} \phi F^2(\nabla u) dm
\end{equation}
for all non-negative functions $\phi\in \mathcal{C}^\infty_0(M)$. Hence, by Lemma \ref{Mean} with $p=1$, $R=\frac{1}{4}$, $\delta=\frac{1}{2}$, we obtain
\be
\sup_{B_{\frac{1}{8}}(x)}F^2(\nabla u)\leq e^{C(1+\frac{a}{4}+\frac{K}{16})}m(B_{\frac{1}{4}})^{-1}\left(1+\frac{K}{16}\right) ^{\frac{\nu}{2}}\int_{B_{\frac{1}{4}}(x)} F^2(\nabla u)dm,\label{mean1}
\ee
where $C=C(n, b, \kappa, \kappa^*)$ is a universal constant.

In the following, we continue to denote some universal constant  by $C>0$ , which may be different line by line. Let $\phi$ be a cut-off function defined by
$$
\phi(z)=
\begin{cases}
1 & \text{ on } B_{\frac{1}{4}}(x), \\
2-4d_{F}(x, z) & \text{ on } B_{\frac{1}{2}}(x) \backslash B_{\frac{1}{4}}(x), \\
0 & \text{ on } B_{{1}}(x) \backslash B_{\frac{1}{2}}(x).
\end{cases}
$$
Then $F^*(-d\phi)\leq 4$ and $F^*(d\phi)\leq 4\kappa$  on $B_{\frac{1}{2}}(x)$. It is easy to see that
\beq
\int_{M}\phi^2 F^2(\nabla u) dm &=&\int_{M}\phi^{2} du(\nabla u) dm =\int_{M}d(u\phi^2)(\nabla u)dm-2\int_{M}u\phi d\phi(\nabla u)dm \nonumber \\
&=& -\int_{M}u \phi^{2} \Delta u \ dm - 2\int_{M} u\phi \ d\phi(\nabla u) \ dm \leq 2 \int_{M} u\phi F^{*}(-d \phi) F(\nabla u) dm \nonumber \\
&\leq &~\frac{1}{2}\int_{M}\phi^2 F^2(\nabla u) dm +2\int_{M}u^2 F^{*2}(-d\phi) dm, \label{addprf}
\eeq
where we have used that $\Delta u =0$ in second line and Young inequality in the last line. (\ref{addprf}) implies that
$$
\int_{M}\phi^2 F^2(\nabla u) dm\leq 4\int_{M}u^2 F^{*2}(-d\phi) dm.
$$
Therefore,
\be
\int_{B_{\frac{1}{4}}(x)}F^2(\nabla u) dm\leq 64\int_{B_{\frac{1}{2}}(x)}  u^2 dm\leq 64~m\left({B_{\frac{1}{2}}(x)}\right)\left(\sup\limits_{B_{\frac{1}{2}}(x)} u\right)^2. \label{Bnau}
\ee
Combining (\ref{Bnau}) with (\ref{mean1}) yields
\be
\sup\limits_{B_{\frac{1}{8}}(x)}F^2(\nabla u)\leq e^{C(1+\frac{3}{2}a+\frac{K}{16})}\left(1+\frac{K}{16}\right)^{\frac{\nu}{2}}\left(\sup\limits_{B_{\frac{1}{2}}(x)}u\right)^{2},
\ee
where we have used volume comparison (\ref{volume}). Further, by applying the Theorem \ref{Harnack}, we can obtain
$$
\sup\limits_{B_{\frac{1}{8}}(x)}F^2(\nabla u)\leq e^{C(1+\frac{3}{2}a+ K)}\left(1+\frac{K}{16}\right)^{\frac{\nu}{2}}\left(\inf\limits_{B_{\frac{1}{8}}(x)}u\right)^2.
$$
Hence, for any $x\in M$, we have
$$
F(\nabla u(x))\leq e^{C(1+\frac{3}{2} a+ K)}\left(1+\frac{K}{16}\right)^{\frac{\nu}{4}}u(x),
$$
where $\nu=4(n+4b)-2$. Finally, we conclude the following
\be
F(\nabla \log u(x))\leq e^{C(1+\frac{3}{2}a+ K)}\left(1+\frac{K}{16}\right)^{n+4b}.
\ee
For $F(\nabla (-\log u))=\overleftarrow{F}(\overleftarrow{\nabla}\log u)$, the same argument works. Thus we conclude that the gradient estimate (\ref{grad}) holds.
\end{proof}

\vskip 10mm

\noindent Xinyue Cheng \\
School of Mathematical Sciences\\
Chongqing Normal University \\
Chongqing, 401331, P.R. China\\
E-mail: chengxy@cqnu.edu.cn

\vskip 5mm

\noindent Liulin Liu \\
School of Mathematical Sciences\\
Chongqing Normal University \\
Chongqing, 401331, P.R. China\\
E-mail:  2023010510010@stu.cqnu.edu.cn

\vskip 5mm

\noindent Yu Zhang  \\
School of Mathematical Sciences\\
Chongqing Normal University \\
Chongqing, 401331, P.R. China\\
E-mail:  2496524954@qq.com

\end{document}